\documentclass[a4paper,12pt]{article}
\usepackage[left=3cm, right=3cm, top=2.5cm, bottom=2.5cm, bindingoffset=0cm]{geometry}
\usepackage[shortlabels]{enumitem}
\usepackage[title]{appendix}
\usepackage{mathtools}
\usepackage[all, cmtip]{xy}
\usepackage[T2A]{fontenc}		
\usepackage[utf8]{inputenc}			
\usepackage{dirtytalk}
\usepackage{subfigure}
\usepackage{comment} 
\usepackage{commath}
\usepackage[final]{epsfig}
\usepackage{graphics}
\usepackage{amsmath}
\usepackage{amsfonts}
\usepackage{amsthm}
\usepackage{float}
\usepackage{latexsym}
\usepackage{amssymb, mathabx}
\usepackage{pifont}
\usepackage{authblk}
\usepackage{epstopdf, setspace}
\usepackage{multicol,multirow, enumitem}
\usepackage{hyperref, mathtools}
\newlist{longenum}{enumerate}{5}
\setlist[longenum, 1]{label=\roman*)}
\setlist[longenum, 2]{label=\alph*)}
\usepackage{array}
\usepackage{wrapfig}
\usepackage{calrsfs}
\DeclareMathAlphabet{\pazocal}{OMS}{zplm}{m}{n}
\usepackage{tikz}
\usetikzlibrary{arrows,intersections}
\usetikzlibrary{positioning}
\usetikzlibrary{decorations.text}
\usetikzlibrary{decorations.pathmorphing}
\usetikzlibrary{decorations.markings}
\usetikzlibrary{patterns}
\usepackage{pgfplots}
\usetikzlibrary{arrows} 

\tikzset{->-/.style={decoration={
markings,
mark=at position .7 with {\arrow{>}}},postaction={decorate}}}
\tikzstyle{vertex} = [coordinate]

\newtheorem{lemma}{Lemma}[section]
\newtheorem{proposition}[lemma]{Proposition}
\newtheorem{remark-definition}[lemma]{Remark-Definition}
\newtheorem{theorem}[lemma]{Theorem}
\newtheorem{corollary}[lemma]{Corollary}

\newtheorem{proposition-conjecture}[lemma]{Proposition-conjecture}

\theoremstyle{definition}
\newtheorem{example}[lemma]{Example}

\newtheorem{definition}[lemma]{Definition}
\newtheorem{remark}[lemma]{Remark}

\newcommand{\eps}{{\varepsilon}}

\newcommand{\R}{{\mathbb R}}

\newcommand{\SDiff}{\mathrm{SDiff}}
\newcommand{\SVect}{\mathfrak{s}\mathfrak{vect}}

\newcommand{\Ker}[1]{\mathrm{Ker} \, #1}

\newcommand{\Ann}[1]{\mathrm{Ann} \, #1}

\newcommand{\diff}[1]{\mathrm{d}  #1}

\newcommand{\Id}{\mathrm{E}}

\newcommand{\Hom}{\mathrm{H}}

\newcommand{\Imm}[1]{\mathrm{Im} \, #1}

\newcommand{\Ad}{\mathrm{Ad}}

\newcommand{\Ham}{\mathrm{Ham}}

\newcommand{\oneform}{\alpha}
\newcommand{\oneformtwo}{\beta}
\newcommand{\oneformthree}{\gamma}
\newcommand{\circulation}{ \pazocal C}

\newcommand{\Diff}{\mathfrak{curl}}

\mathtoolsset{showonlyrefs}

\title{Classification of coadjoint orbits for symplectomorphism groups of surfaces}
\author{Ilia Kirillov\thanks{Department of Mathematics,
University of Toronto, Toronto, ON M5S 2E4, Canada;
e-mail: {\tt kirillov@math.utoronto.ca}
}}
\date{}
\pgfplotsset{compat=1.14} 
\begin{document}
\maketitle
\begin{abstract} 
We classify generic coadjoint orbits for symplectomorphism groups of compact symplectic surfaces with or without boundary. We also classify simple Morse functions on such surfaces up to a symplectomorphism.
\end{abstract}
\tableofcontents 
\section{Introduction}
\label{section:Intro}
The classification problem for coadjoint orbits for the action of symplectic (or area-preserving) diffeomorphisms in two dimensions was known to specialists in view of its application in fluid dynamics since the 1960s, and it was explicitly formulated in~\cite[see Section I.5]{arnold1999topological} in 1998. 
The same classification problem also arises in Poisson geometry since coadjoint orbits are symplectic leaves of the Lie-Poisson bracket, and also in representation theory in connection with the orbit method of A.Kirillov~\cite{kirillov1993orbit}. In the recent work~\cite{penna2020sdiff} the orbit method was applied to the symplectomorphism group of the two-sphere. The classification of generic coadjoint orbits was obtained in~\cite{izosimov2017classification, izosimov2016coadjoint} for the case of closed surfaces. In~\cite{izosimov2017classification} there is a list of open questions for the case of surfaces with boundary. We answer all those questions in Sections~\ref{section:global} and \ref{section:orbits}.
\par
The classification problem for coadjoint orbits for symplectomorphisms of a surface is closely related to a certain classification problem for functions, specifically, with each coadjoint orbit one can associate a (\emph{vorticity}) function up to a symplectomorphism
(see details in Section~\ref{subsection:orbits_to_functions}). Hence we are going to address the following two problems:
\begin{enumerate}
    \item Classify generic coadjoint orbits of symplectomorphism groups of surfaces.
    \item Classify generic smooth functions on symplectic surfaces up to symplectomorphisms.
\end{enumerate}
There is a number of papers devoted to the classification of functions with non-degenerate critical points on surfaces. In~\cite{hladysh2019simple} simple Morse functions on surfaces with boundary were classified with respect to smooth left-right equivalence. In~\cite{izosimov2016coadjoint} simple Morse functions on closed compact symplectic surfaces  were classified up to a symplectomorphism. We generalize the results of~\cite{izosimov2016coadjoint} to the case of surfaces with boundary. The classification of functions is given in Theorem~\ref{classification_of_functions_symplectic}. Roughy speaking, the classification theorem for functions states that there a one-to-one correspondence between functions up to a symplectomorphism and measured Reeb graphs up to an isomorphism. The classification of coadjoint orbits is given in Theorem~\ref{classification_of_orbits}. The classification result for coadjoint orbits in also given in terms of measured Reeb graphs supplemented with some additional data (so called augmented circulation Reeb graphs).

It is worth noting that the classification of functions in~\cite{izosimov2016coadjoint} is based on the classification of so-called simple Morse fibrations obtained in~\cite{dufour1994classification}; on the contrary, the proof in the present paper uses a different method so it gives an alternative proof for Theorem~3.11 from~\cite{izosimov2016coadjoint} (classification of functions in the case of closed surfaces).
\par
This paper is organised as follows: in Section~\ref{section:local} we give a local classification of Morse functions up to symplectomorphisms. In Section~\ref{section:global} we solve the global classification problem of simple Morse functions. In Section~\ref{section:orbits} we use the results of Section~\ref{section:global} in order to classify generic coadjoint orbits of the symplectomorphism group and illustrate these results with examples. In Section~\ref{section:Final_remarks} we discuss some open problems in this area, and in Appendix~\ref{section:euler} we present a hydrodynamical motivation for the main classification theorem. 
\subsection*{Acknowledgements}
The author is grateful to A.Izosimov, B.Khesin, E.Kudryavtseva, V.Matveev, A.Oshemkov, A. Prishlyak, and the anonymous referee for discussions and valuable comments, which significantly improved this paper. This research is supported in part by the Russian Science Foundation (grant No. 17-11-01303) and the Simons Foundation.
\section{Local classification of functions up to a symplectomorphism}
\label{section:local}
\subsection{Preliminaries}
Throughout this section, let $(M,\omega)$ be a compact connected symplectic surface with the boundary $\partial M.$
\begin{definition}[\cite{jankowski1972functions}]\label{def:simpleMorse}
A \emph{smooth} function $F \colon M \to \R$ is called a \emph{simple Morse} function if the following conditions hold:
\begin{enumerate}[label=(\roman*)]
\item
All critical points of $F$ are non-degenerate;
\item
$F$ does not have critical points on the boundary $\partial M$; 
\item
The restriction of $F$ to the boundary $\partial M$ is a Morse function;
\item
All critical values of $F$ and of its restriction $F|_{\partial M}$ are distinct.
\end{enumerate}
\end{definition}
\begin{proposition}[\cite{jankowski1972functions}]
    Simple Morse functions form an open and dense subset in the space of smooth functions in the $C^2$-topology. 
\end{proposition}
For a regular point $O\in M\setminus\partial M$ of the function $F$ there exists a coordinate chart $(p,q)$ centered at $O$ such that $F=p$ and $\omega=dp\wedge dq.$ For a regular point $O\in \partial M$ of the restriction $F|_{\partial M}$ there exists a coordinate chart $(p,q)$ in $M$ centered at $O$ such that $F=p,$ $\omega=dp\wedge dq$ and the boundary $\partial M$ is given by the equation $\{q=0\}.$ Next, we present the normal forms for the pair $(F,\omega)$ in a neighbourhood of a critical point for the function $F$ and of its restriction $F|_{\partial M}.$
\par 
\subsection{The case of a singular point inside the surface}
\begin{theorem}[\cite{CDV1979morse,toulet1996thesis}]\label{theorem:md}
Let $(M,\omega)$ be a symplectic surface, and $F\colon M\to\R$ be a simple Morse function. Let $O\in M\setminus\partial M$ be a critical point for the function $F.$ Then there exists a coordinate chart $(p,q)$ centered at $O$ such that $\omega=dp\wedge dq$ and $F=\lambda\circ S$ where $S=pq$ or $S=p^2+q^2.$ Here $\lambda$ is a smooth function of one variable defined in some neighborhood of the origin $0\in \R$ and $\lambda'(0) \ne 0.$
Moreover:
\begin{enumerate}[label=(\roman*)]
\item In the case $S=p^2+q^2,$ 
and the function $\lambda$ is uniquely determined by the pair $(F,\omega).$
\item In the case $S=pq,$ only the Taylor series of  the function $\lambda$ is uniquely determined by the pair $(F,\omega).$ 
In other words, if $(\tilde{p},\tilde{q})$ is another chart as above then $\tilde{p}\tilde{q} = pq+\psi(pq)$ (for sufficiently small $p,q,\tilde{p},\tilde{q}),$ where $\psi$ is a function of one variable flat\footnote{Here flat means that all derivatives of $\psi$ vanish at the origin.} 
at the origin. Furthermore, every function of one variable that is flat at the origin can be obtained in this way.
\end{enumerate}
\end{theorem}
\begin{proof}
All statements of this theorem but the last one are proved in~\cite{CDV1979morse,toulet1996thesis}. The last statement is proved in~\cite[see Lemma~$3.2a$]{toulet1996thesis}.
\end{proof}
\subsection{The case of a singular point on the boundary} 
\begin{theorem}\label{theorem:md_boundary}
Let $(M,\omega)$ be a symplectic surface, and $F\colon M\to\R$ be a simple Morse function. 
Let $O\in \partial M$ be a regular point of $F$ and a non-degenerate critical point of its restriction $F|_{\partial M}.$
Then there exists a coordinate chart $(p, q)$ centered at $O$ such that $\omega=dp\wedge dq$ and $F=\lambda\circ S$ where $S=q+p^2$ or $S=q-p^2.$ Here $\lambda$ is a smooth function of one variable defined in some neighborhood of the origin $0\in \R$ and $\dot{\lambda}(0) \ne 0.$ In this chart $M$ is defined by $q\ge 0$ and the boundary $\partial M$ is given by the equation $\{q=0\},$ see Figure~\ref{levelsetsboundary}. Moreover:
\begin{enumerate}[label=(\roman*)]
\item In the case $S=q+p^2,$
and the function $\lambda$ is uniquely determined by the pair $(F,\omega).$
\item In the case $S=q-p^2,$ only the Taylor series of  the function $\lambda$ is uniquely determined by the pair $(F,\omega).$ 
In other words, if $(\tilde{p},\tilde{q})$ is another chart as above then $\tilde{q}-\tilde{p}^2=q-p^2+\psi(q-p^2).$
Furthermore, every function of one variable that is flat at the origin can be obtained in this way.
\end{enumerate}
\end{theorem}
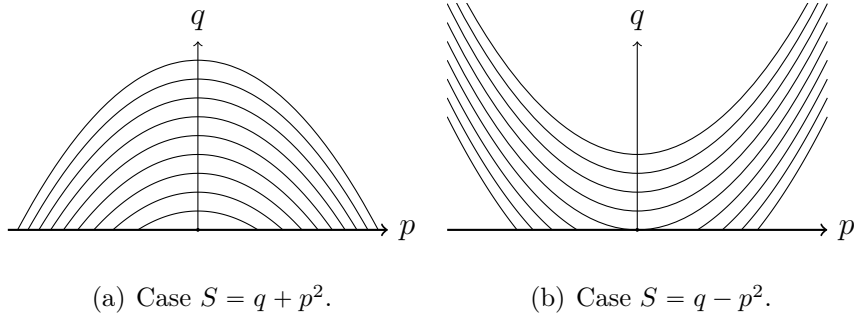
\begin{figure}[H]
\centering
\subfigure[Case $S=q+p^2.$]{
	\begin{tikzpicture}[scale=2.5]
	\filldraw (0,0) circle (0.2pt);
	\clip (-1,-0.2) rectangle (1.2,1.2);
	\draw[->, thick] (-1,0) -- (1,0) node[right] {$p$};
	\draw[->] (0,0) -- (0,1) node[above] {$q$};
	\foreach \x in {0.0, 0.1, 0.2, 0.3, 0.4, 0.5, 0.6, 0.7,0.8, 0.9}
		\draw[domain={-sqrt(\x)}:{sqrt(\x)}, smooth, variable=\y] plot ({\y},{\x-\y*\y});
	\end{tikzpicture}}
\subfigure[Case $S=q-p^2.$]{
	\begin{tikzpicture}[scale=2.5]
	\filldraw (0,0) circle (0.2pt);
	\clip (-1,-0.2) rectangle (1.2,1.2);
	\draw[->, thick] (-1,0) -- (1,0) node[right] {$p$};
	\draw[->] (0,0) -- (0,1) node[above] {$q$};
	\foreach \x in {0.4, 0.3, 0.2, 0.1, 0.0}
		\draw  (-1,\x+1) parabola bend (0,\x) (1,\x+1);
	\foreach \x in {0.4, 0.3, 0.2, 0.1}
		\draw[domain={sqrt(\x)}:1, smooth, variable=\y] plot ({\y},{\y*\y-\x});
	\foreach \x in {0.4, 0.3, 0.2, 0.1}
		\draw[domain=-1:{-sqrt(\x)}, smooth, variable=\z] plot ({\z},{\z*\z-\x});
	\end{tikzpicture}}
\caption{Level sets of the function $S$. The horizontal axis corresponds to the boundary $\partial M.$}
\label{levelsetsboundary}
\end{figure}
Before we proceed with the proof of this theorem let us formulate and prove a lemma.
\begin{lemma}\label{lemma_sqrt_diff}
Let $h_1$ and $h_2$ be two smooth non-negative functions $\R_+\to\R_+$ such that $h_i(0)=0$ and $\dot{h}_i(0)>0$ for $i=1,2.$ Then the following statements are equivalent:
\begin{enumerate}[label=(\roman*)]
\item The difference $h_1-h_2$ is a function flat at the origin, i.e. the Taylor series $J^\infty_0h_1$ and $J^\infty_0h_2$ are equal to each other.
\item The difference 
$
\sqrt{h_1}-\sqrt{h_2}
$ 
is a smooth function $\R_+\to\R.$
\item
The difference 
$
\sqrt{h_1}-\sqrt{h_2}
$ 
is a smooth function $\R_+\to\R$ flat at the origin. 
\end{enumerate}
\end{lemma}
\begin{proof}
The implication $(iii)\implies (ii)$ is evident so it enough to show that 
$
(i)\implies (iii)
$
and 
$
(ii)\implies (i).
$
Let us start with implication $(i)\implies (iii).$ It follows from Hadamard's lemma that there exist smooth functions $\tilde{h}_1$ and $\tilde{h}_2$ such that 
$
h_i=x\tilde{h}_i
$ 
and 
$\tilde{h}_i(0)>0$
for 
$
i=1,2.
$
We have the following formula for the difference 
$
\sqrt{h_1}-\sqrt{h_2}:
$
\begin{equation}
\label{eqn_sqrt_diff}
\begin{aligned}
\sqrt{h_1(x)}-\sqrt{h_2(x)}
=\frac{1}{\sqrt{x}}\frac{h_1(x)-h_2(x)}{\sqrt{\tilde{h}_1(x)}+\sqrt{\tilde{h}_2(x)}}
\end{aligned}   
\end{equation}
for small enough $x.$
It follows from the formula \eqref{eqn_sqrt_diff} that the difference $\sqrt{h_1}-\sqrt{h_2}$ is smooth and flat at the origin whenever the difference $h_1-h_2$ is flat the origin. 
\par
It remains to show that $(ii)$ implies (i). Denote by $g$ the smooth function $\sqrt{h_1}-\sqrt{h_2}.$ Assume that the difference $h_1-h_2$ is not flat at the origin. Then there exists a number $n\in\mathbb{N}$ and a smooth non-zero function $f:\R_+\to \R\setminus0$ such that 
$
h_1(x)-h_2(x)=x^n f(x).
$
It is useful to rewrite formula \eqref{eqn_sqrt_diff} in the following form:
\begin{equation}
\label{eqn_sqrt_diff_alt}
\begin{aligned}
g(x)\sqrt{x}\left(\sqrt{\tilde{h}_1(x)}+\sqrt{\tilde{h}_2(x)}\right)
=x^n f(x).
\end{aligned}   
\end{equation}
Formula \eqref{eqn_sqrt_diff_alt} implies that the function $f$ is flat at the origin whenever the function $g$ is flat at the origin. The function $f$ is a non-zero function so we conclude that the function $g$ is not flat at the origin. Therefore there exists a number $m\in\mathbb{N}$ and a smooth non-zero function $\tilde{g}:\R_+\to \R\setminus0$ such that 
$
g(x)=x^m\tilde{g}(x).
$
Now we take the square of both sides of \eqref{eqn_sqrt_diff_alt} and obtain the following formula:
\begin{equation}
\begin{aligned}
x^{1+2m}\tilde{g}(x)^2\left(\sqrt{\tilde{h}_1(x)}+\sqrt{\tilde{h}_2(x)}\right)^2
=x^{2n} f(x)^2.
\end{aligned}   
\end{equation}
That gives us a contradiction since the Taylor series of the left hand side starts with an odd power of $x$ and the Taylor series of the right hand side starts with an even power of $x.$ We conclude that the function 
$
h_1-h_2
$ 
is flat at the origin.
\end{proof}
Now we proceed with the proof of Theorem~\ref{theorem:md_boundary}.
\begin{proof}[Proof of Theorem~\ref{theorem:md_boundary}]
Without loss of generality we can assume that $F(O)=0.$ The main part of this theorem on the existence of a coordinate chart was proved in~\cite{kirillov2018morse, kourliouros2019local}.
\par
Let us prove statement $(i)$ of this theorem. We need to prove the equality $\tilde{p}+\tilde{q}^2=p+q^2.$  In this case (see Figure \ref{levelsetsboundary}, (a)) the region $\{F\leq \varepsilon \}$ is diffeomorphic to a closed half ball provided that $\varepsilon>0$ is sufficiently small. Therefore, the area of this region 
$$
A_{F,\omega}(\varepsilon):=\int_{F\leq\varepsilon}\omega
$$
is well-defined. Let $(p,q)$ be a coordinate chart centered at $O$ such that 
$
F=\lambda(p^2+q)$ and $\omega =\diff p \wedge \diff q.
$ 
Then we can write an explicit formula for the function 
$
A_{F,\omega}(\varepsilon)\colon
$
$$
A_{F,\omega}(\varepsilon)
=\int_{-\sqrt{\lambda^{-1}(\varepsilon)}}^{\sqrt{\lambda^{-1}(\varepsilon)}} (\lambda^{-1}(\varepsilon) - p^2)\diff p
=\frac{4}{3}\lambda^{-1}(\varepsilon)^{3/2}.
$$
So we conclude that the function $\lambda$ (and thus the function $q+p^2$) is uniquely determined by the pair $(F,\omega).$
\par
Now let us prove statement $(ii)$ of this theorem.  Consider the second case where in some local chart centered at $O$ we have $F=\lambda(q-p^2)$ and 
$
\omega =\diff p \wedge \diff q.
$
Consider a smooth curve $\mathit{\ell} \subset M$ such that it is transversal to the right half of the parabola $\{q=p^2\}$ and these two curves intersect each other at exactly one point. 
In coordinates $(p,q)$ the curve $\mathit{\ell}$ can be described as a graph of some function $g\colon$
$$
\mathit{\ell}=\{(p,q)\colon q=g(p)\}.
$$
We fix a number $\varepsilon<0,$ and consider the region $R_{\varepsilon}$ (see Figure~\ref{fig:area_function}) bounded by the boundary curve
$
\partial M=\{q=0\},
$ 
the right half of the parabola 
$
\{q=p^2\}\cap\{p\geq 0\},
$ 
the right half of the parabola 
$
\{q=p^2+\lambda^{-1}(\varepsilon)\}\cap\{p\geq 0\},
$ 
and the curve $\mathit{\ell}.$
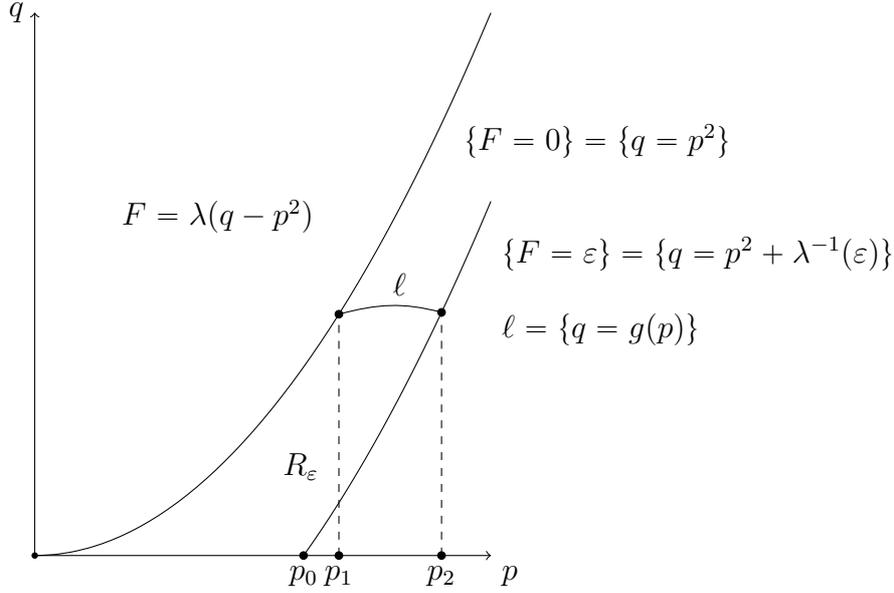
\begin{figure}[H]
    \centering
    \begin{tikzpicture}[scale=5]
	\filldraw (0,0) circle (0.2pt);
	\node [right] at (1.1,1.1) {$\{F=0\}=\{q=p^2\}$};
	\node [right] at (1.2,0.8) {$\{F=\varepsilon\}=\{q=p^2+\lambda^{-1}(\varepsilon)\}$};
	\node [right] at (1.2,0.6) {$\mathit{\ell}=\{q=g(p)\}$};
	\node [right] at (0.2,0.9) {$F=\lambda(q-p^2)$};
	\node [below] at (0.707,0) {$p_0$};
	\draw[fill] (0.707,0) circle [radius=0.01];
	\node [below] at (0.7,0.3) {$R_{\varepsilon}$};
	\node [below] at (0.8,0) {$p_1$};
	\draw[fill] (0.8,0) circle [radius=0.01];
	\node [below] at (1.07,0) {$p_2$};
	\draw[fill] (1.07,0) circle [radius=0.01];
	\node [above] at (0.96,0.66) {$\mathit{\ell}$};
	\draw[->] (0,0) -- (1.2,0) node[below right] {$p$};
	\draw[->] (0,0) -- (0,1.2*1.2) node[left] {$q$};
	\draw[domain=0:1.2, smooth, variable=\y] plot ({\y},{\y*\y});
	\draw[domain={sqrt(0.5)}:1.2, smooth, variable=\y] plot ({\y},{\y*\y-0.5});
	\draw (0.8,0.8*0.8) to [out=15,in=165] (1.07,1.07*1.07-0.5);
	\draw [dashed] (0.8,0) -- (0.8,0.8*0.8);
	\draw[fill] (0.8,0.8*0.8) circle [radius=0.01];
	\draw [dashed] (1.07,0) -- (1.07,1.07*1.07-0.5);
	\draw[fill] (1.07,1.07*1.07-0.5) circle [radius=0.01];
	\end{tikzpicture}
\caption{The area function 
$
A_{F,\omega,\mathit{\ell}}=
\int_{R_{\varepsilon}}\omega.
$
}
\label{fig:area_function}
\end{figure}
Then the area of this region 
$$
A_{F,\omega,\mathit{\ell}}(\varepsilon):=\int_{R_{\varepsilon}}\omega
$$
is well-defined. Denote by $p_0$ the $p$-coordinate of the intersection $\{F=\varepsilon\}\cap \partial M,$ by $p_1$ the $p$-coordinate of the intersection $\{F=0\}\cap \mathit{\ell},$
and by $p_2$ the $p$-coordinate of the intersection $\{F=\varepsilon\}\cap \mathit{\ell}.$ 
The coordinates $p_0,$ $p_1,$ and $p_2$ depend on $\mathit{\ell},$ and also the coordinates $p_0$ and $p_2$ depend on $\varepsilon.$ 
It follows from the implicit function theorem that the coordinate $p_2$ is a smooth function of $\varepsilon.$ As for $p_0,$ it is explicitly given by $\sqrt{-\lambda^{-1}(\varepsilon)}.$ Note that 
$
p_0(\varepsilon)<p_1<p_2(\varepsilon)
$ 
provided that $\abs{\varepsilon}$ is sufficiently small.
We have the following formula for the function 
$
A_{F,\omega,\mathit{\ell}}(\varepsilon)\colon
$
\begin{equation}
\begin{aligned}
& A_{F,\omega,\mathit{\ell}}(\varepsilon)=
\int_0^{p_0(\varepsilon)}p^2\diff p-\int_{p_0(\varepsilon)}^{p_1}\lambda^{-1}(\varepsilon)\diff p+\int_{p_1}^{p_2(\varepsilon)}(g(p)-p^2-\lambda^{-1}(\varepsilon))\diff p\\
& =\,p_0^3(\varepsilon)/3-p_0(\varepsilon)\lambda^{-1}(\varepsilon)+\mbox{smooth function of }\varepsilon\\
& =\,\frac{4}{3}(-\lambda^{-1}(\varepsilon))^{3/2}+\mbox{smooth function of }\varepsilon.
\end{aligned}
\end{equation}
Now consider some other chart $(\tilde{p},\tilde{q})$ such that $F=\tilde{\lambda}(q-p^2),$ 
$\omega=\diff\tilde{p}\wedge\diff \tilde{q},$ and the boundary $\partial M$ is given by $\{\tilde{q}=0\}.$ Then it is follows from above that
\begin{equation}
\begin{aligned}\label{eqn_lambda_diff}
\frac{4}{3}(-\lambda^{-1}(\varepsilon))^{3/2}-\frac{4}{3}(-\tilde{\lambda}^{-1}(\varepsilon))^{3/2}=f(\varepsilon)
\end{aligned}
\end{equation}
where $f$ is a smooth function of one variable. We want to prove that the Taylor series $J^\infty_0\lambda$ is equal to the Taylor series of $J^\infty_0\tilde{\lambda}.$ It follows from Lemma~\ref{lemma_sqrt_diff} that the Taylor series 
$
J^\infty_0\lambda^{-1}(\varepsilon)^3
$
is equal to the Taylor series 
$
J^\infty_0\tilde{\lambda}^{-1}(\varepsilon)^3.
$
From here we conclude that 
$
J^\infty_0\lambda=J^\infty_0\tilde{\lambda}.
$
\par
It remains to prove the last part of statement $(ii)$. Let $\psi\colon\R\to\R$ be a function of one variable flat at the origin.
The goal is to find a symplectomorphism $\Phi$ defined in some neighbourhood of the origin such that $O$ is a fixed point for $\Phi,$ the symplectomorphism $\Phi$ preserves the boundary $\partial M,$ and
$$
\Phi^*[q+\psi(q)]=q.
$$
For this part of the proof we are going to use a different coordinate chart (see Figure \ref{parabolic}):
$$
(P:=p,Q:=q-p^2).
$$
First of all, notice that after this \say{parabolic} change of coordinates the symplectic form still has the standard form:
$$
\diff P \wedge \diff Q = \diff p \wedge \diff q.
$$
Secondly, in the chart $(P,Q) $the function $F$ \say{straightens} i.e. $F(P,Q)=Q+\psi(Q),$ 
and the boundary becomes a parabola $\partial M=\{(P+Q^2=0\}.$ 
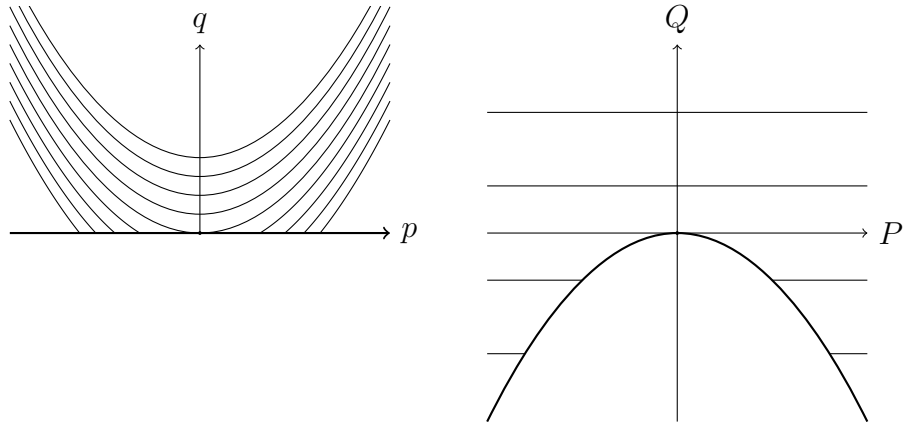
\begin{figure}[H]
\centering
\subfigure[Chart $(p,q)$]{
	\begin{tikzpicture}[scale=2.5]
	\filldraw (0,0) circle (0.2pt);
	\clip (-1.2,-1.2) rectangle (1.2,1.2);
	\draw[->, thick] (-1,0) -- (1,0) node[right] {$p$};
	\draw[->] (0,0) -- (0,1) node[above] {$q$};
	\foreach \x in {0.4, 0.3, 0.2, 0.1, 0.0}
		\draw  (-1,\x+1) parabola bend (0,\x) (1,\x+1);
	\foreach \x in {0.4, 0.3, 0.2, 0.1}
		\draw[domain={sqrt(\x)}:1, smooth, variable=\y] plot ({\y},{\y*\y-\x});
	\foreach \x in {0.4, 0.3, 0.2, 0.1}
		\draw[domain=-1:{-sqrt(\x)}, smooth, variable=\z] plot ({\z},{\z*\z-\x});
	\end{tikzpicture}}
\subfigure[Chart $(P,Q)$]{
	\begin{tikzpicture}[scale=2.5]
	\filldraw (0,0) circle (0.2pt);
	\clip (-1.2,-1.2) rectangle (1.2,1.3);
	\draw[->] (-1,0) -- (1,0) node[right] {$P$};
	\draw[->] (0,-1) -- (0,1) node[above] {$Q$};
	\draw[thick]  (-1,-1) parabola bend (0,0) (1,-1);
	\draw (-1,0.64) -- (1,0.64);
	\draw (-1,0.25) -- (1,0.25);
	\draw (-1,-0.25) -- (-0.5,-0.25);
	\draw (-1,-0.64) -- (-0.8,-0.64);
	\draw (1,-0.25) -- (0.5,-0.25);
	\draw (1,-0.64) -- (0.8,-0.64);
	\end{tikzpicture}}	
\caption{The parabolic change of coordinates}
\label{parabolic}
\end{figure}
Now let us proceed with the proof. Apply Moser's path method and consider the family of functions 
$$
f^t:=Q+t\psi(Q)
$$
for each $t\in[0,1].$ Instead of looking for one symplectomorphism $\Phi$, we will be looking for a family of Hamiltonian symplectomorphisms $\Phi^t$ such that 
\begin{align}\label{homotopyEqn}
\Phi^{t*}f^t=Q, 
\end{align}
$\Phi^{t}(\partial M)\subset \partial M$ for each $t\in[0,1],$ and $\Phi^{t}(O)=O$ for each $t\in[0,1].$ Let $v^t$ be the vector field corresponding to the flow $\Phi^t:$
$$
\frac{d}{d t}\Phi^t=v^t\circ\Phi^t.
$$
Differentiating \eqref{homotopyEqn} with respect to $t,$ we obtain the following differential equation
$$
\Phi^{t*}L_{v^t}f^t+\Phi^{t*}\frac{df^t}{dt}=0,
$$
which we rewrite as
$$
\Phi^{t*}\left(L_{v^t}f^t+\frac{df^t}{dt}\right)=0.
$$
Since $\Phi^t$ is a diffeomorphism, it is equivalent to   
\begin{align}\label{Dif_Eqn1}
L_{v^t}f^t+\psi(Q)=0.
\end{align}
Since the flow of the field $v^t$ has to preserve the symplectic structure $\omega,$ we will be looking for the field $v^t$ in the Hamiltonian form 
\begin{align}\label{HamiltonianForm}
v^t=H_Q^t\frac{\partial}{\partial P}-H_P^t\frac{\partial}{\partial Q}
\end{align}
where $H_Q^t:=\frac{\partial H^t}{\partial Q}$ and $H_P^t:=\frac{\partial H^t}{\partial P}.$
Substitute the right-hand side of \eqref{HamiltonianForm} into \eqref{Dif_Eqn1} to obtain the following partial differential equation
$$
\psi(Q)-H^t_P(1+t\dot{\psi}(Q))=0.
$$
Rewrite it as
$$
H^t_P=-\frac{\psi(Q)}{1+t\dot{\psi}(Q)}.
$$
Consider the family of functions 
$$
\psi^t(x):=\frac{\psi(x)}{1+t\dot{\psi}(x)}
$$
for each $t\in[0,1].$ We have $\dot{\psi}(0)=0$ so the denominator $1+t\dot{\psi}(x)$ is non-zero
in sufficiently small neighbourhood of the origin. Then our equation assumes the form 
\begin{align}\label{Dif_Eqn4_}
H^t_P=\psi^t(Q).
\end{align} 
It is clear that the general solution to this equation has the form 
$$
H^t(P,Q)=P\psi^t(Q)+g^t(Q)
$$ 
where $g^t$ is a smooth function of one variable. Our goal is to find a particular solution to \eqref{Dif_Eqn4_} that is constant along the boundary $\partial M=\{Q+P^2=0\}.$ That implies the following condition on the function $g^t\colon$
$$
P\psi^t(-P^2)+g^t(-P^2)=0.
$$
Hence, for any non-positive $x\in\R$ we have
$$
g^t(x)=-\sqrt{-x}\psi^t(x).
$$
Now define a function $g^t$ in the following way:
\[ \begin{cases} 
      g^t(x)=-\sqrt{-x}\psi^t(x) & x\leq 0 \\
      g^t(x)=0 & x>0
   \end{cases}
\]
It follows from the flatness of $\psi^t(\cdot)$ that the function $g^t$ defined as above is a smooth function flat at the origin. Now define the function $H$ to be the corresponding solution:
$$
H^t(P,Q)=P\psi^t(Q)+g^t(Q)
$$
The family of symplectomorphisms $\Phi^t$ can be recovered as the flow of the corresponding field 
$
v^t=H_Q^t\frac{\partial}{\partial P}-H_P^t\frac{\partial}{\partial Q}.
$ 
The condition 
$H^t|_{\partial M}=0$ 
implies that the field $v^t$ has a zero restriction on the boundary 
$\partial M = \{Q+P^2=0\},$ 
and we conclude that the corresponding family $\Phi^t$ preserves the boundary, and $\Phi^t(O)=O$ for each $t\in[0,1].$
Now applying the theorem on the smooth dependence of the flow on initial data one can conclude that the flow $\Phi^t$ is well-defined for $t\in[0,1].$ Hence, the diffeomorphism $\Phi^1$ has the desired properties. 
\end{proof}
In the next section we are going to use these local results to obtain a global classification of simple Morse functions with respect to the action of the group $\SDiff(M)$ of symplectomorphisms of $M.$
\section{Global classification of functions on symplectic surfaces}
\label{section:global}
\subsection{The Reeb graph of a function}
Throughout this section, let $M$ be a compact connected oriented surface with boundary $\partial M,$ and let $F \colon M\to \R$ be a simple Morse function on $M$. In what follows, by a level we mean a connected component of level sets of $F.$ Non-critical levels are diffeomorphic to a circle or a line segment. The surface $M$ can be considered as a union of levels, and we get a foliation with singularities. 
The base space of this foliation with the quotient topology is homeomorphic to a finite connected graph $\Gamma_F$ (see Figure~\ref{fig:example_reeb}) whose vertices correspond to critical values of $F$ or $F|_{\partial M}.$ We view this graph as a topological object (rather than combinatorial). This graph $\Gamma_F$ is called the \emph{Reeb graph}\footnote{This graph is also called the Kronrod graph of a function, see~\cite{Reeb, Kronrode}.
}
of the function $F.$ By $\pi$ we denote the projection $M\to \Gamma_F.$
We denote the edges of the Reeb graph by \emph{solid} lines if they correspond to circle components and by \emph{dashed} lines if they correspond to segment components. We denote the union of solid (respectively, dashed) edges in $\Gamma_F$ by $\Gamma_F^s$ and $\Gamma_F^d,$ respectively. We denote the preimages $\pi^{-1}(\Gamma^s_F)$ and $\pi^{-1}(\Gamma^d_F)$ by $M^s_F$ and $M^d_F.$ Thus $\Gamma_F = \Gamma_F^s \cup \Gamma_F^d,$ and $M = M^s_F \cup M^d_F.$ There are 7 possible types of vertices in the graph $\Gamma_F$ (see Table~\ref{table:types}). The function $F$ on $M$ descends to a function $f$ on the Reeb graph $\Gamma_F$. It is also convenient
to assume that $\Gamma_F$ is oriented: edges are oriented in the direction of increasing $f.$
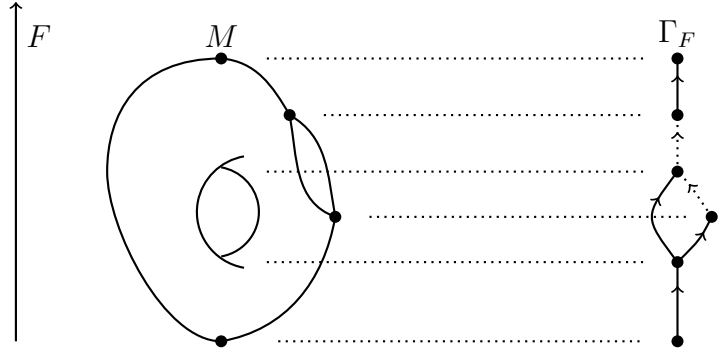
\begin{figure}[H]
\centering
{
\begin{tikzpicture}[thick, scale=1.5]
\node [vertex] at (1,2) (nodeDi) {};
\node [vertex] at (1.6,1.5) (nodeDDi) {};
\node [vertex] at (1,-0.5) (nodeAi) {};
\node [vertex] at (2, 0.6) (nodeBBi) {};
\fill (nodeAi) circle [radius=1.5pt];
\fill (nodeBBi) circle [radius=1.5pt];
\fill (nodeDi) circle [radius=1.5pt];
\fill (nodeDDi) circle [radius=1.5pt];
\draw (nodeDi) to [out=-5, in=120] (nodeDDi);
\draw (nodeDDi) to [out=-80, in=160] (nodeBBi);
\draw (nodeDDi) to [out=-30, in=100] (nodeBBi);
\draw (1,-0.5) to [out=10, in=-100] (nodeBBi);
\draw (0,1) .. controls (0,1.5) and (0.25,2) .. (1,2)
(nodeAi) .. controls (0.5,-0.5) and (0,0.5) .. (0,1);
\draw   (1.2,0.15) arc (260:100:0.5cm);
\draw   (1,0.25) arc (-80:80:0.4cm);
\node [vertex] at (5,-0.5) (nodeA) {$A$};
\node [vertex] at (5,0.2) (nodeB) {B};
\node [vertex] at (5.3,0.6) (nodeBB) {BB};
\draw  [->-] (nodeA) -- (nodeB);
\node [vertex] at (5,1) (nodeC) {C};
\draw  [->-] (nodeB) .. controls +(-0.3,+0.4) .. (nodeC);
\draw  [->-] (nodeB) .. controls +(0.2,+0.2) .. (nodeBB);
\draw  [dotted, ->-] (nodeBB) .. controls +(-0.2,+0.3) .. (nodeC);
\node [vertex] at (5,2) (nodeD) {D};
\node [vertex] at (5,1.5) (nodeDD) {DD};
\draw  [dotted, ->-] (nodeC) -- (nodeDD);
\draw  [->-] (nodeDD) -- (nodeD);
\fill (nodeA) circle [radius=1.5pt];
\fill (nodeB) circle [radius=1.5pt];
\fill (nodeBB) circle [radius=1.5pt];
\fill (nodeC) circle [radius=1.5pt];
\fill (nodeD) circle [radius=1.5pt];
\fill (nodeDD) circle [radius=1.5pt];
\draw  [->] (-0.8,-0.5) -- (-0.8,2.5);
\node at (-0.6, 2.2) (nodeA) {$F$};
\draw  [dotted] (1.5,-0.5) -- (4.7,-0.5);
\draw  [dotted] (1.4,0.2) -- (4.7,0.2);
\draw  [dotted] (2.3, 0.6) -- (5.1, 0.6);
\draw  [dotted] (1.4,1) -- (4.7,1);
\draw  [dotted] (1.9,1.5) -- (4.7,1.5);
\draw  [dotted] (1.4,2) -- (4.7,2);
\draw (5,2) node[above] {$\Gamma_F$};
\draw (1,2) node [above] {$M$};
\end{tikzpicture}} 
\caption{A torus with one hole with the height function on it and the corresponding Reeb graph}\label{fig:example_reeb}
\end{figure}
\par
Let $v$ be a vertex of the Reeb graph $\Gamma_F.$ 
Let us fix a number $\varepsilon>0$ such that 
$$
f^{-1}([f(v)-\varepsilon,f(v)+\varepsilon])\cap e
$$
is a proper subset of $e$ for each edge $e$ incident to $v.$ Consider the preimage 
$
P_v^\varepsilon:=\pi^{-1}(f^{-1}([f(v)-\varepsilon,f(v)+\varepsilon]))\subset M. 
$
The boundary $\partial \pi^{-1}(P^\varepsilon_v)$ is a piecewise smooth closed oriented curve. This curve is connected in the case where the vertex $v$ is incident only to dashed edges, and its image $\pi[\partial \pi^{-1}(P^\varepsilon_v)]$ is a closed oriented curve that passes edges incident to the vertex $v$ in a certain cyclic order. This construction is nontrivial only in the case when there are at least three dashed edges incident to the vertex $v$ (otherwise, there is only one cyclic order at the set of edges incident to $v$). Thus for an arbitrary II-vertex or IV-vertex (see Table~\ref{table:types}) of the graph $\Gamma_F$ we have a natural cyclic order for the edges incident to this vertex. The above properties of the graph $\Gamma_F$ make it natural to introduce the
following definition of an abstract Reeb graph.
\begin{definition}\label{def:Reeb_graph}
    An (abstract) Reeb graph $(\Gamma, f)$ is an oriented connected graph $\Gamma$ with solid or dashed edges, and a continuous function $f:\Gamma \to \R,$ with the following properties and additional data:
\begin{enumerate}[label=(\roman*)]
\item Each vertex of $\Gamma$ is of one of the 7 types from Table~\ref{table:types}.
\item There is a cyclic order on the set of edges incident to II- or IV-vertices (see Table~\ref{table:types})
\item The function $f$ is strictly monotonic on each edge of $\Gamma$, and the edges of $\Gamma$ are oriented towards the direction of increasing $f$.
\end{enumerate}
\end{definition}
\begin{definition}
Abstract Reeb graphs $(\Gamma_1,f)$ and $(\Gamma_2,g)$ are said to be \emph{equivalent} by means of the isomorphism $\phi\colon\Gamma_1\to\Gamma_2$ if the map $\phi:$
\begin{enumerate}[label=(\roman*)]
\item maps solid (respectively, dashed) edges to solid (respectively, dashed) edges;
\item preserves the cyclic order on the set of edges incident to each $II$- or $IV$-vertex, i.e. if $e_2$ follows $e_1$ in the cyclic order, then $\phi(e_2)$ follows $\phi(e_1);$
\item takes the function $g$ to the function $f$ (i.e. $f=g\circ \phi$).
\end{enumerate}
\end{definition}
\subsection{Recovering the topology of a surface from the Reeb graph}
In this subsection we follow~\cite[Section 5]{hladysh2019simple}. Let $M$ be a compact connected oriented surface with the boundary $\partial M,$ and let $F \colon M\to \R$ be a simple Morse function. The restriction of the projection $\pi$ to each boundary component of $M$ is a closed curve (a map from a circle to the graph) in the graph $\Gamma_F.$ Informally speaking, the following definition describes those closed curves for an abstract Reeb graph.
\begin{definition}\label{def:boundary_cycle}
Let $(\Gamma, f)$ be an abstract Reeb graph. A non-empty sequence of edges $(e_1,e_2,\dots,e_n)$ together with a sequence $(v_1,v_2,\dots,v_n,v_{n+1}=v_1)$ of vertices is called a \emph{boundary cycle} if the following three conditions hold:
\begin{enumerate}[label=(\roman*)]
    \item All edges in the sequence are dashed.
    \item Each edge $e_i$ is incident to the vertices $v_i$ and $v_{i+1}$ for every $i \in \{1,\dots, n\}.$
    \item If the vertex $v_{i}$ has three or more adjacent dashed edges, then the pair ($e_{i-1},e_i)$ of consecutive edges is also a consecutive pair of edges with respect to the cyclic order on the set of edges incident to the vertex $v_{i}$ for every $i \in \{1,\dots, n\}.$
\end{enumerate}
We call two boundary cycles equivalent if they differ by the action of a cyclic group, i.e. the sets of vertices
$
v_1\dots v_n v_1
$
and
$
v_i\dots v_n v_1 \dots v_{i-1}v_i
$
define the same topological cycle for each $i\in\{1,\dots,n\}.$ In addition, in the case when a boundary cycle consists only of $1$ or $2$-valent vertices (i.e. of vertices of type III and IV) we also call two boundary cycles 
$
v_1 v_2 \dots v_{n-1} v_n v_1
$
and 
$
v_n v_{n-1}\dots v_2 v_1 v_n
$
equivalent. We denote by $\sigma(\Gamma)$ the number of (equivalence classes of) boundary cycles in $\Gamma.$ 
\end{definition}
\begin{example}
Consider a disk with holes and a torus with one hole, and consider the height function on them (as shown in Figure~\ref{fig:dashed_Reeb}). The corresponding Reeb graphs are identical except for the cyclic orders at the vertices $C_1$ and $C_2.$ In case $(a)$ of a disk with with holes there are three boundary cycles: $B_1C_1D_1B_1,$ $C_1E_1D_1C_1,$ and $A_1 B_1 D_1 E_1 F_1 E_1 C_1 B_1 A_1.$ In case $(b)$ of a torus with one hole there is only one boundary cycle: $A_2 B_2 D_2 E_2 F_2 E_2 C_2 D_2 B_2 C_2 E_2 D_2 C_2 D_2 A_2.$ 
\end{example}
\begin{proposition}
[{\cite[page 12]{hladysh2019simple}}]\label{prop:boundary_cycle}
Let $M$ be a compact connected oriented surface with the boundary $\partial M,$ and let $F \colon M\to \R$ be a simple Morse function.
Then the number of boundary cycles $\sigma(\Gamma_F)$ is equal to the number of boundary components $\dim H_0(\partial M)$ of the surface $M.$    
\end{proposition}
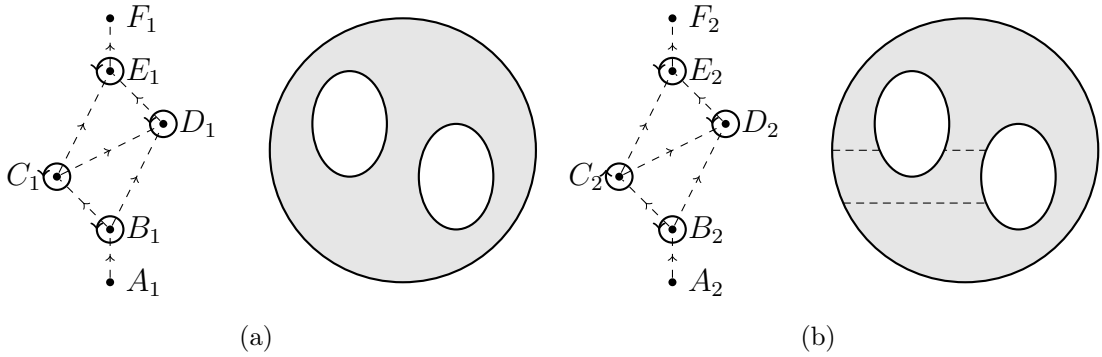
\begin{figure}
\centering
\subfigure[]{
\begin{tikzpicture}[scale = 0.7]
    \tikzstyle{subj} = [circle, minimum width=3pt, fill, inner sep=0pt]
    \node[subj, label=right:$A_1$] (nodeA) at (5.5,0) {};
    \node[subj, label=right:$B_1$] (nodeB) at (5.5,1) {};
    \node[subj, label=left:$C_1$] (nodeC) at (4.5,2) {};
    \node[subj, label=right:$D_1$] (nodeD) at (6.5,3) {};
    \node[subj, label=right:$E_1$] (nodeE) at (5.5,4) {};
    \node[subj, label=right:$F_1$] (nodeF) at (5.5,5) {};
    \draw[thick, decoration={markings, mark=at position 0.5 with {\arrow{>}}}, postaction={decorate}]
    (nodeB) circle (0.25);
    \draw[thick, decoration={markings, mark=at position 0.5 with {\arrow{>}}}, postaction={decorate}]
    (nodeC) circle (0.25);
    \draw[thick, decoration={markings, mark=at position 0.5 with {\arrow{>}}}, postaction={decorate}]
    (nodeD) circle (0.25);
    \draw[thick, decoration={markings, mark=at position 0.5 with {\arrow{>}}}, postaction={decorate}]
    (nodeE) circle (0.25);
    \draw[dashed, 
    decoration={markings, mark=at position 0.5 with {\arrow{>}}},
    postaction={decorate}]
    (nodeA) -- (nodeB);
    \draw[dashed, 
    decoration={markings, mark=at position 0.5 with {\arrow{>}}},
    postaction={decorate}]
    (nodeB) -- (nodeC);
    \draw[dashed, 
    decoration={markings, mark=at position 0.5 with {\arrow{>}}},
    postaction={decorate}] 
    (nodeB) -- (nodeD);
    \draw[dashed, 
    decoration={markings, mark=at position 0.5 with {\arrow{>}}},
    postaction={decorate}] 
    (nodeC) -- (nodeE);
    \draw[dashed, 
    decoration={markings, mark=at position 0.5 with {\arrow{>}}},
    postaction={decorate}] 
    (nodeC) -- (nodeD);
    \draw[dashed, 
    decoration={markings, mark=at position 0.5 with {\arrow{>}}},
    postaction={decorate}] 
    (nodeD) -- (nodeE);
    \draw[dashed, 
    decoration={markings, mark=at position 0.5 with {\arrow{>}}},
    postaction={decorate}] 
    (nodeE) -- (nodeF);
    \fill [opacity = 0.1] (1+10, 2.5) ellipse (2.5 and 2.5);
    \draw [thick] (1+10, 2.5) ellipse (2.5 and 2.5);
    \fill [thick, fill = white, draw = black] (1+9, 3) ellipse (0.7 and 1);
    \fill [thick, fill = white, draw = black] (1+11, 2) ellipse (0.7 and 1);
\end{tikzpicture}}
\subfigure[]{
\begin{tikzpicture}[scale = 0.7]
    \tikzstyle{subj} = [circle, minimum width=3pt, fill, inner sep=0pt]
    \node[subj, label=right:$A_2$] (nodeA) at (5.5,0) {};
    \node[subj, label=right:$B_2$] (nodeB) at (5.5,1) {};
    \node[subj, label=left:$C_2$] (nodeC) at (4.5,2) {};
    \node[subj, label=right:$D_2$] (nodeD) at (6.5,3) {};
    \node[subj, label=right:$E_2$] (nodeE) at (5.5,4) {};
    \node[subj, label=right:$F_2$] (nodeF) at (5.5,5) {};
    \draw[thick, decoration={markings, mark=at position 0.5 with {\arrow{>}}}, postaction={decorate}]
    (nodeB) circle (0.25);
    \draw[thick, decoration={markings, mark=at position 0.5 with {\arrow{<}}}, postaction={decorate}]
    (nodeC) circle (0.25);
    \draw[thick, decoration={markings, mark=at position 0.5 with {\arrow{>}}}, postaction={decorate}]
    (nodeD) circle (0.25);
    \draw[thick, decoration={markings, mark=at position 0.5 with {\arrow{>}}}, postaction={decorate}]
    (nodeE) circle (0.25);
    \draw[dashed, 
    decoration={markings, mark=at position 0.5 with {\arrow{>}}},
    postaction={decorate}]
    (nodeA) -- (nodeB);
    \draw[dashed, 
    decoration={markings, mark=at position 0.5 with {\arrow{>}}},
    postaction={decorate}]
    (nodeB) -- (nodeC);
    \draw[dashed, 
    decoration={markings, mark=at position 0.5 with {\arrow{>}}},
    postaction={decorate}] 
    (nodeB) -- (nodeD);
    \draw[dashed, 
    decoration={markings, mark=at position 0.5 with {\arrow{>}}},
    postaction={decorate}] 
    (nodeC) -- (nodeE);
    \draw[dashed, 
    decoration={markings, mark=at position 0.5 with {\arrow{>}}},
    postaction={decorate}] 
    (nodeC) -- (nodeD);
    \draw[dashed, 
    decoration={markings, mark=at position 0.5 with {\arrow{>}}},
    postaction={decorate}] 
    (nodeD) -- (nodeE);
    \draw[dashed, 
    decoration={markings, mark=at position 0.5 with {\arrow{>}}},
    postaction={decorate}] 
    (nodeE) -- (nodeF);
    \fill [opacity = 0.1] (1+10, 2.5) ellipse (2.5 and 2.5);
    \draw [thick] (1+10, 2.5) ellipse (2.5 and 2.5);
    \draw [densely dashed] (1+7.5, 2.5) -- (1+11.5, 2.5);
    \draw [densely dashed] (1+7.7, 1.5) -- (1+11.5, 1.5);  
    \fill [thick, fill = white, draw = black] (1+9, 3) ellipse (0.7 and 1);
    \fill [thick, fill = white, draw = black] (1+11, 2) ellipse (0.7 and 1);
\end{tikzpicture}}
\caption{An illustration to Definition~\ref{def:boundary_cycle}: dashed Reeb graph with $\dim \Hom_1(\Gamma) = 2$ corresponding to both a disk with two holes $(a)$ and torus with one hole $(b).$ Cutting the disk drawn here along the three dashed levels and then restoring the gluings with opposite orientations, one obtains a torus with one hole. This figure is based on Figure~5 from~\cite{izosimov2017classification}.}\label{fig:dashed_Reeb}
\end{figure}
\begin{theorem}[{\cite[Theorem 5.3]{hladysh2019simple}}]\label{formula_for_total_genus}
The genus $g(M)$ of a surface $M$ is given by the following formula:
\begin{equation}
\label{formula:genus}
\begin{aligned}
&g(M)= -\chi(\Gamma^s)+\frac{-\chi(\Gamma^d)+5\dim \Hom_0(\Gamma^s\cap \Gamma^d)-\sigma(\Gamma_F)}{2}\\
&\,\,\,\,\,\,\,\,\,\,\,\,\,\,\,\,\,\,\,\,-\,\dim \Hom_0(\Gamma^s)-\dim \Hom_0(\Gamma^d)+3,
\end{aligned}
\end{equation}
where $\chi(\Gamma^s)$ is the Euler characteristic and $\sigma(\Gamma_F)$ is the number of boundary cycles.
\end{theorem}
Theorem~\ref{formula_for_total_genus} motivates us to give the following definition. 
\begin{definition}\label{def:genus_graph}
Let $(\Gamma, f)$ be an abstract Reeb graph. Define the \emph{genus} $g(\Gamma)$ as the number from the right-hand side of the formula in Theorem~\ref{formula_for_total_genus}.
\end{definition}
\begin{table*}
{
\begin{tabular}{m{1cm}|m{4cm}|m{2cm}|m{8cm}|}
\centering{Type} & \centering{Level Sets} & \centering{Reeb Graph} & \centering{Asymptotics} \tabularnewline
\hline
\centering I 
& 
\centering
\begin{tikzpicture}[thick, scale = 1.33]
  \draw (-0.91, 0) -- (0.91, 0);
  \draw [dotted] (0.3, 0) .. controls (0.2, 0.3) and (-0.2, 0.3) .. (-0.3,0);
    \draw [dotted](0.6, 0) .. controls (0.4, 0.6) and (-0.4, 0.6) .. (-0.6,0);
      \draw [dotted] (0.9, 0) .. controls (0.6, 0.9) and (-0.6, 0.9) .. (-0.9,0);
        \fill [opacity = 0.1] (-0.9, 0) -- (0.9, 0)  .. controls (0.6, 0.9) and (-0.6, 0.9) .. (-0.9,0) ;
            \fill (0, 0) circle [radius=1.5pt];
      \end{tikzpicture}
&
\centering
\begin{tikzpicture}[thick, scale = 2]
      \draw[dashed, decoration={markings, mark=at position 0.5 with {\arrow{>}}}, postaction={decorate}](1.5, 0) -- (1.5, 0.5);%
       \fill (1.5, 0) circle [radius=1pt];
\end{tikzpicture}
&
$\mu([v,x])= \psi(  f(x) )\sqrt{|   f (x)|}, $
where $\psi(0) = 0,$ and $\dot{\psi}(0) \neq 0$.
\\
\hline
\centering II 
& 
\centering
\begin{tikzpicture}[thick, scale = 1.33]
  \draw [] (-0.64, 0) -- (0.64, 0);
                     \fill (0, 0) circle [radius=1.5pt];
            \begin{scope}
    \clip (-1.1,0) rectangle (1.1, 0.8);
      \draw [dotted] (0, 0.8) circle [radius = 0.8];
     \draw [dotted] (0, 0.8) circle [radius = 1];
       \fill [opacity = 0.1] (0, 0.8) circle [radius = 1];
              \fill [fill = white, draw = white] (0, 0.8) circle [radius = 0.6];
                       \draw [dotted] (0, 0.8) circle [radius = 0.6];
\end{scope}
  \draw [] (-1.01, 0.8) -- (-0.6, 0.8);
  \draw [] (1.01, 0.8) -- (0.6, 0.8);   
\end{tikzpicture}
& 
\centering
\begin{tikzpicture}[thick, scale = 2]
              \draw[dashed, decoration={markings, mark=at position 0.5 with {\arrow{>}}}, postaction={decorate}](1.2, 0) -- (1.5, 0.5) node[left, align=left] at  (1.35,0.25)
              {$e_0$};
                  \draw[dashed, decoration={markings, mark=at position 0.5 with {\arrow{>}}}, postaction={decorate}](1.8, 0) -- (1.5, 0.5) node[right, align=right] at (1.65, 0.25) {$e_1$};
         \draw[dashed, decoration={markings, mark=at position 0.5 with {\arrow{>}}}, postaction={decorate}](1.5, 0.5) -- (1.5, 1) node[right, align=right] at (1.5,0.7) {$e_2$};
       \fill (1.5, 0.5) circle [radius=1pt];
\end{tikzpicture}
& 
$\mu([v,x])= \,\eps_i\psi(  f(x) )\sqrt{|   f (x)|} + \eta_i(   f(x) ),$
where $\eps_0 = \eps_1 = -1,$ $\eps_2 =  2,$  $\psi(0) = 0,$ $\dot{\psi}(0) \neq 0,$
and $\eta_0 + \eta_1 + \eta_2 = 0$.
\\
\hline
\centering III
&
\centering
\begin{tikzpicture}[thick, scale = 2]
  \draw [] (-0.32, 0) -- (0.32, 0);
  \draw [dotted] (0, 0.4) circle [radius = 0.4];
          \fill (0, 0) circle [radius=1pt];
            \begin{scope}
    \clip (-0.6,0) rectangle (0.6, 1);
     \draw [dotted] (0, 0.4) circle [radius = 0.5];
       \fill [opacity = 0.1] (0, 0.4) circle [radius = 0.5];
              \fill [ fill = white] (0, 0.4) circle [radius = 0.3];
              \draw [dotted]  (0, 0.4) circle [radius = 0.3];;
\end{scope}
\end{tikzpicture}
&
\centering
\begin{tikzpicture}[thick, scale = 2]
     \draw[dashed, decoration={markings, mark=at position 0.5 with {\arrow{>}}}, postaction={decorate}](1.5, 0) -- (1.5, 0.5) node[right, align=right] at (1.5, 0.2) {$e_0$};
         \draw[decoration={markings, mark=at position 0.5 with {\arrow{>}}}, postaction={decorate}](1.5, 0.5) -- (1.5, 1) node[right, align=right] at (1.5, 0.7) {$e_1$};
       \fill (1.5, 0.5) circle [radius=1pt];
\end{tikzpicture}
& 
$\mu([v,x])= \,\eps_i\psi(  f(x) )\sqrt{|   f (x)|} + \eta_i(   f(x) ),$
where $\eps_0 = -1$, $\eps_1  = 0$,  $\psi(0) = 0$, $\dot{\psi}(0) \neq 0$,
and $\eta_0 + \eta_1 = 0$.
\\
\hline
\centering  IV
&
\centering
\begin{tikzpicture}[thick, scale = 1.33]
                     \fill (0, 0) circle [radius=1.5pt];
                        \begin{scope}                  
                         \clip (-0.65, -1) rectangle (0.65, 1);
                                              \begin{pgfinterruptboundingbox}
                     \draw [dotted] (0,0) .. controls (-1.5,1.5) and (-1.5,-1.5) .. (0,0);
                     \end{pgfinterruptboundingbox}
                         \begin{pgfinterruptboundingbox}
                                          \draw [dotted] (0,0) .. controls (1.5,1.5) and (1.5,-1.5) .. (0,0);
                                          \end{pgfinterruptboundingbox}
                     \draw [dotted] plot [smooth cycle] coordinates {(0, 0.2) (-0.65, 0.55) (-1.1, 0.4) (-1.25, 0) (-1.1, -0.4) (-0.65, -0.55) (0,-0.2) (0.65, -0.55) (1, -0.5) (1.2,0) (1, 0.5)(0.65, 0.55)};
                       \fill [opacity = 0.1] plot [smooth cycle] coordinates {(0, 0.2) (-0.65, 0.55) (-1.1, 0.4) (-1.25, 0) (-1.1, -0.4) (-0.65, -0.55) (0,-0.2) (0.65, -0.55) (1, -0.5) (1.2,0) (1, 0.5)(0.65, 0.55)};
                         \fill [fill = white] (0.65,0) ellipse (0.35 and 0.3);
                         \draw [dotted] (0.65,0) ellipse (0.35 and 0.3);;
                           \fill [fill = white] (-0.65,0) ellipse (0.35 and 0.3);
                            \draw [dotted] (-0.65,0) ellipse (0.35 and 0.3);;
                     \end{scope}                    
           \draw [] (0.65, 0.285) -- (0.65, 0.56);
                                   \draw [] (0.65, -0.285) -- (0.65, -0.56);
                               \draw [] (-0.65, 0.285) -- (-0.65, 0.56);
                                   \draw [] (-0.65, -0.285) -- (-0.65, -0.56);
\end{tikzpicture}
& 
\centering
\begin{tikzpicture}[thick, scale = 2]
              \draw [dashed, decoration={markings, mark=at position 0.5 with {\arrow{>}}}, postaction={decorate}](1.2, 0) -- (1.5, 0.5) node[left, align=left] at (1.35, 0.25) {$e_0$};
                  \draw [dashed, decoration={markings, mark=at position 0.5 with {\arrow{>}}}, postaction={decorate}](1.8, 0) -- (1.5, 0.5) node[right, align=right] at (1.65, 0.25) {$e_1$};
         \draw [dashed, decoration={markings, mark=at position 0.5 with {\arrow{>}}}, postaction={decorate}](1.5, 0.5) -- (1.2, 1) node[left, align=left] at (1.35, 0.75) {$e_3$};
             \draw[dashed, decoration={markings, mark=at position 0.5 with {\arrow{>}}}, postaction={decorate}](1.5, 0.5) -- (1.8, 1) node[right, align=right] at (1.65, 0.75) {$e_2$};
       \fill (1.5, 0.5) circle [radius=1pt];
\end{tikzpicture}
&
$\mu([v,x])= \,\eps_i\psi(  f(x) )\ln |   f (x)| + \eta_i(   f(x) ),$
where $\eps_0 = \eps_1= -1$, $\eps_2  = \eps_3 =  1$,  $\psi(0) = 0$, $\dot{\psi}(0) \neq 0$,
and $\eta_0 + \eta_1 + \eta_2 + \eta_3 = 0$.
\\
\hline
\centering V
& 
\begin{tikzpicture}[thick, scale = 1.33]
                     \fill (0, 0) circle [radius=1.5pt];
                     \begin{pgfinterruptboundingbox}
                     \draw [dotted] (0,0) .. controls (-1.5,1.5) and (-1.5,-1.5) .. (0,0);
                     \end{pgfinterruptboundingbox}
                        \begin{scope}
                         \clip (-1.3, -1) rectangle (0.65, 1);
                         \begin{pgfinterruptboundingbox}
                                          \draw [dotted] (0,0) .. controls (1.5,1.5) and (1.5,-1.5) .. (0,0);
                                          \end{pgfinterruptboundingbox}
                     \draw [dotted] plot [smooth cycle] coordinates {(0, 0.2) (-0.65, 0.55) (-1.1, 0.4) (-1.25, 0) (-1.1, -0.4) (-0.65, -0.55) (0,-0.2) (0.65, -0.55) (1, -0.5) (1.2,0) (1, 0.5)(0.65, 0.55)};
                       \fill [opacity = 0.1] plot [smooth cycle] coordinates {(0, 0.2) (-0.65, 0.55) (-1.1, 0.4) (-1.25, 0) (-1.1, -0.4) (-0.65, -0.55) (0,-0.2) (0.65, -0.55) (1, -0.5) (1.2,0) (1, 0.5)(0.65, 0.55)}; 
                         \fill [fill = white] (0.65,0) ellipse (0.35 and 0.3);
                         \draw [dotted] (0.65,0) ellipse (0.35 and 0.3);
                     \end{scope}
                         \fill [fill = white] (-0.65,0) ellipse (0.35 and 0.3);
                         \draw [dotted] (-0.65,0) ellipse (0.35 and 0.3);;
                          \draw [] (0.65, 0.285) -- (0.65, 0.56);
                                   \draw [] (0.65, -0.285) -- (0.65, -0.56);
\end{tikzpicture}
&
\centering
\begin{tikzpicture}[thick, scale = 2]
              \draw[decoration={markings, mark=at position 0.5 with {\arrow{>}}}, postaction={decorate}](1.2, 0) -- (1.5, 0.5) node[left, align=left] at (1.35, 0.25) {$e_0$};
                  \draw [dashed, decoration={markings, mark=at position 0.5 with {\arrow{>}}}, postaction={decorate}](1.8, 0) -- (1.5, 0.5) node[right, align=right] at (1.65, 0.25) {$e_1$};
         \draw[dashed, decoration={markings, mark=at position 0.5 with {\arrow{>}}}, postaction={decorate}](1.5, 0.5) -- (1.5, 1) node[right, align=right] at (1.5, 0.7) {$e_2$};
       \fill (1.5, 0.5) circle [radius=1pt];
\end{tikzpicture}
& 
$\mu([v,x])= \,\eps_i\psi(  f(x) )\ln |   f (x)| + \eta_i(   f(x) ),$
where $\eps_0 = \eps_1 =  -1$, $\eps_2 = 2,$ $\psi(0) = 0$, $\dot{\psi}(0) \neq 0$,
and $\eta_0 + \eta_1 + \eta_2 = 0$.
\\
\hline
\centering VI
&
\centering 
\begin{tikzpicture}[thick, scale = 1.33]
                     \fill (0, 0) circle [radius=1.5pt];
                        \begin{scope}                  
                         \clip (-1.5, -1) rectangle (1.5, 1);
                                              \begin{pgfinterruptboundingbox}
                     \draw [dotted] (0,0) .. controls (-1.5,1.5) and (-1.5,-1.5) .. (0,0);
                     \end{pgfinterruptboundingbox}
                         \begin{pgfinterruptboundingbox}
                                          \draw [dotted] (0,0) .. controls (1.5,1.5) and (1.5,-1.5) .. (0,0);
                                          \end{pgfinterruptboundingbox}
                     \draw [dotted] plot [smooth cycle] coordinates {(0, 0.2) (-0.65, 0.55) (-1.1, 0.4) (-1.25, 0) (-1.1, -0.4) (-0.65, -0.55) (0,-0.2) (0.65, -0.55) (1, -0.5) (1.2,0) (1, 0.5)(0.65, 0.55)};
                       \fill [opacity = 0.1] plot [smooth cycle] coordinates {(0, 0.2) (-0.65, 0.55) (-1.1, 0.4) (-1.25, 0) (-1.1, -0.4) (-0.65, -0.55) (0,-0.2) (0.65, -0.55) (1, -0.5) (1.2,0) (1, 0.5)(0.65, 0.55)};
                         \fill [fill = white] (0.65,0) ellipse (0.35 and 0.3);
                         \draw [dotted] (0.65,0) ellipse (0.35 and 0.3);;
                           \fill [fill = white] (-0.65,0) ellipse (0.35 and 0.3);
                            \draw [dotted] (-0.65,0) ellipse (0.35 and 0.3);;
                     \end{scope}                    
\end{tikzpicture}
& 
\centering
\begin{tikzpicture}[thick, scale = 2]
              \draw[decoration={markings, mark=at position 0.5 with {\arrow{>}}}, postaction={decorate}](1.2, 0) -- (1.5, 0.5) node[left,align=left] at (1.35, 0.25) {$e_0$};
                  \draw[decoration={markings, mark=at position 0.5 with {\arrow{>}}}, postaction={decorate}](1.8, 0) -- (1.5, 0.5) node[right,align=right] at (1.65, 0.25) {$e_1$};
         \draw[decoration={markings, mark=at position 0.5 with {\arrow{>}}}, postaction={decorate}](1.5, 0.5) -- (1.5, 1) node[right,align=right] at (1.5, 0.7) {$e_2$};
       \fill (1.5, 0.5) circle [radius=1pt];
\end{tikzpicture}
&
$\mu([v, x])= \,\eps_i\psi(  f(x) )\ln |   f (x)| + \eta_i(   f(x) ),$
where $\eps_0 = \eps_1 =  -1,$ $\eps_2 = 2,$ $\psi(0) = 0$, $\dot{\psi}(0) \neq 0$,
and $\eta_0 + \eta_1 + \eta_2 = 0$.
\\
\hline
\centering VII
& 
\centering
\begin{tikzpicture}[thick, scale = 2]
  \draw [dotted] (0, 0.4) circle [radius = 0.4];
          \fill (0, 0.4) circle [radius=1pt];
            \begin{scope}
    \clip (-0.6,-0.12) rectangle (0.6, 1);
     \draw [dotted] (0, 0.4) circle [radius = 0.5];
       \fill [opacity = 0.1] (0, 0.4) circle [radius = 0.5];
              \draw [dotted]  (0, 0.4) circle [radius = 0.2];
              \draw [dotted]  (0, 0.4) circle [radius = 0.3];;
\end{scope}
\end{tikzpicture}
& 
\centering
\begin{tikzpicture}[thick, scale = 2]
      \draw [decoration={markings, mark=at position 0.5 with {\arrow{>}}}, postaction={decorate}](1.5, 0) -- (1.5, 0.5);
       \fill (1.5, 0) circle [radius=1pt];
\end{tikzpicture}
&
$\mu([v, x])= \psi(f(x)),$ where $\psi(0) = 0,$ and $\dot{\psi}(0) \neq 0.$
\\
\hline
\end{tabular}
}
\caption{7 types of neighborhoods of singular points with corresponding Reeb graphs and asymptotics for the measure on a Reeb graph (figures are partially taken from~\cite{izosimov2017classification}). The notation $\mu([v, x])$ is a measure that is introduced below in Definition~$3.8.$ 
In order to simplify notation we assume that $f(v)=0.$ If not, we replace $f$ by $\tilde{f}(x):=f(x)-f(v).$
}
\label{table:types}
\end{table*}
\subsection{Measured Reeb graphs}
Now, fix an area form $\omega$ on the surface $M.$ Then the natural projection map $\pi \colon M \to \Gamma_F$ induces a measure $\mu:=\pi_*\omega$ on the graph $\Gamma_F.$ 
\begin{definition}\label{def:quasi-smooth_measure}
A measure $\mu$ on an abstract Reeb graph $(\Gamma, f)$ is called \emph{quasi-smooth} if the following conditions hold.
\begin{enumerate}
\item The measure $\mu$ has a $C^\infty$-smooth non-zero density $\diff \mu / \diff f$ in the complement $\Gamma\setminus V(\Gamma)$.
\item In a neighbourhood of each vertex the measure $\mu$ can be expressed by the corresponding formula from Table~\ref{table:types}.
\end{enumerate}
\end{definition}
\begin{proposition}
Let $(M,\omega)$ be a compact connected symplectic surface with a boundary $\partial M,$ and let $F \colon M\to \R$ be a simple Morse function. Then the measure $\mu:=\pi_*\omega$ is quasi-smooth.
\end{proposition}
\begin{proof}
    For vertices of types~VI~and~VII this was proved in~\cite[Subsection I.1.2]{dufour1994classification}. The proof is based on Theorem~\ref{theorem:md}, the essence of the proof is the study of the area between the non-singular level sets of the function $F$ and a singular $F$-level. The proof for other types follows the same lines, with the only difference that it uses both Theorems~\ref{theorem:md}~and~\ref{theorem:md_boundary}. Note that or vertices of types~I~and~VII the function 
    $\psi$ is uniquely (and explicitly) determined by the corresponding function $\lambda$ (see Theorems~\ref{theorem:md}~and~\ref{theorem:md_boundary}). In other cases $\psi$ is determined by the corresponding function $\lambda$ up to a function flat at the origin, and there is no explicit expression for $\psi$ in terms of $\lambda$ (see details in Toulet's thesis \cite[Subsection 2.2]{toulet1996thesis}.
\end{proof}
The above properties of the measure $\mu$ make it natural to introduce the
following definition of an abstract measured Reeb graph.
\begin{definition}\label{def:measured_Reeb_graph}
    A \emph{measured Reeb graph} $(\Gamma, f,\mu)$ is a Reeb graph $(\Gamma, f)$ equipped with a quasi-smooth measure $\mu.$
\end{definition}
\begin{definition}
    Two measured Reeb graphs $(\Gamma_1,f,\mu)$ and $(\Gamma_2,g,\nu)$ are said to be \emph{equivalent} by means of the isomorphism $\phi:\Gamma_1\to\Gamma_2$ if the map $\phi:$
\begin{enumerate}[label=(\roman*)]
\item is an isomorphism between the Reeb graphs $(\Gamma_1,f)$ and $(\Gamma_2,g);$
\item pushes the measure $\mu$ to the measure $\nu.$
\end{enumerate}
\end{definition}
\begin{definition}\label{def:compatible}
{\rm
A measured Reeb graph $(\Gamma, f, \mu)$ is \emph{compatible} with $(M, \omega)$ if the following conditions hold:
\begin{enumerate}[label=(\roman*)]  
\item The genus $g(\Gamma)$ of the graph $\Gamma$ is equal to the genus $g(M)$ of the surface $M.$
\item The number $\sigma(\Gamma)$ of boundary cycles is equal to the number $\dim H_0(\partial M)$ of boundary components of the surface $M.$
\item The volume of $\Gamma$ with respect to the measure $\mu$ is equal to the area of the surface $M$: $\int_\Gamma \diff \mu=\int_M\omega.$
\end{enumerate}
}
\end{definition}
\subsection{Classification of simple Morse functions up to a symplectomorphism}
\begin{theorem}\label{classification_of_functions_symplectic}
Let $M$ be a compact connected oriented surface with boundary $\partial M.$ Then there is a one-to-one correspondence between simple Morse functions on $M$, considered up to symplectomorphism, and (isomorphism classes of) measured Reeb graphs compatible with $M$. In other words, the following statements hold.
\begin{longenum}
   \item Let $F,$ $G \colon M \to \R$ be two simple Morse functions. Then the following conditions are equivalent:
   \begin{longenum}
       \item There exists a symplectomorphism $\Phi \colon M \to M$ such that $\Phi^*F=G.$
       \item Measured Reeb graphs of $F$ and $G$ are isomorphic. 
   \end{longenum}
   Moreover, every isomorphism $\phi \colon (\Gamma_F, f, \mu_F) \to (\Gamma_G, g, \mu_G)$ can be lifted to a symplectomorphism $\Phi \colon M \to M$ such that $\Phi^*F=G.$
   \item For each measured Reeb graph $(\Gamma, f, \mu)$ compatible  with $(M,\omega)$ there exists a simple Morse function $F:M\to\R$ such that the corresponding measured Reeb graph $\Gamma_F$ is isomorphic to $(\Gamma, f, \mu).$
\end{longenum}
\end{theorem}
\begin{remark}
Note that the formulation of this theorem is identical to the formulation of Theorem~$3.11$ from \cite{izosimov2016coadjoint}.
The difference, of course, is that all notions in the present paper are extended to cover the case of surfaces with boundary.
\end{remark}
\begin{proof}
Let us prove the first statement. The implication $(a) \implies (b)$ is evident, so it suffices to prove the implication $(b) \implies (a).$
Let $\phi \colon \Gamma_F \to \Gamma_G$ be an isomorphism of measured Reeb graphs. We need to construct a symplectomorphism $\Phi \colon M \to M$ such that $\Phi^*F=G$ and $\pi_G\circ\Phi=\phi\circ\pi_F.$
\par
Let $\mathit{\ell} \subset M$ be a smooth oriented curve which is transversal to the level sets of the function $F,$ it does not intersect the singular levels of the function $F,$ and such that the function $F$ is strictly increasing along the curve $\mathit{\ell}.$ Consider the Hamiltonian flow $P_F^t$ corresponding to the function $F.$ We denote by $T_F(p_F,q_F)$ the time necessary to go from the curve $\ell$ to the point $(p_F,q_F)$ under the action of $P^t;$ see Figure~\ref{fig:time_function_a}. The pair of functions $(F,T_F)$ forms a coordinate system in some neighborhood of $\ell$ such that $\omega=\diff F \wedge \diff T_F$ (it is a standard computation, see proof in \cite[Lemma 4]{kirillov2018morse}). In particular, this construction works for the boundary curve $\partial M;$ see Figure~\ref{fig:time_function_b}.
The range of the function $T_F$ along the non-critical level of $F$ is a segment $[0,\Pi(F)]$ in the case when the $F$-level is a segment, and it is a half-interval $[0,\Pi(F))$ in the case when the $F$-level is a circle. The function $\Pi$ is called a period. It follows from Stokes' theorem that $\Pi(F)$ is equal to the derivative $\frac{\diff \mu}{\diff f}.$ 
\begin{figure}[H]
\centering
\subfigure[]{\label{fig:time_function_a}
\begin{tikzpicture}[scale=0.5]
\draw[black] (5,5) node[anchor=north] {$(F,0)$};
\draw[black] (0,5) node[anchor=north] {$(F,T_F)$};
\draw [fill] (5, 5) circle [radius=.1];
\draw [fill] (0, 5) circle [radius=.1];
\draw[black, thick,->] (0, 5) -- (10, 5) node[right] {};
\draw[black, thick,->] (0, 1) -- (10, 1) node[right] {};
\draw[black, thick,->] (0, 9) -- (10, 9) node[right] {};
\draw (0,0) .. controls (4,1) and (6,9) .. (10,10) node[right] {$\mathit{\ell}$};
\draw[black, thick,->] (-2, 0) -- (-2, 10) node[right] {$F$};
\draw [decorate, decoration={brace, amplitude=10pt}, xshift=0pt, yshift=0pt]
(0,5) -- (5,5) node [black, midway, above, yshift=10pt] {\footnotesize $T_F$};
\end{tikzpicture}}
\subfigure[]{\label{fig:time_function_b}
\begin{tikzpicture}[scale=0.5]
\draw[black] (5,5) node[anchor=north] {$(F,T_F)$};
\draw [fill] (5, 5) circle [radius=.1];
\draw [fill] (10, 5) circle [radius=.1];
\draw[black, thick,->] (0, 5) -- (10, 5) node[right] {$(F,0)$};
\draw[black, thick,->] (0, 1) -- (10, 1) node[right] {};
\draw[black, thick,->] (0, 9) -- (10, 9) node[right] {};
\draw[black, thick,->] (0, 0) -- (0, 10) node[right] {$\partial M$};
\draw[black, thick,<-] (10, 0) -- (10, 10) node[right] {$\partial M$};
\draw[black, thick,->] (-2, 0) -- (-2, 10) node[right] {$F$};
\draw [decorate, decoration={brace, amplitude=10pt},xshift=0pt, yshift=0pt]
(5,5) -- (10,5) node [black, midway, above, yshift=10pt] {\footnotesize $T_F$};
\end{tikzpicture}}
\caption{An illustration to the definition of the function $T_F.$}
\label{fig:time_function}
\end{figure}
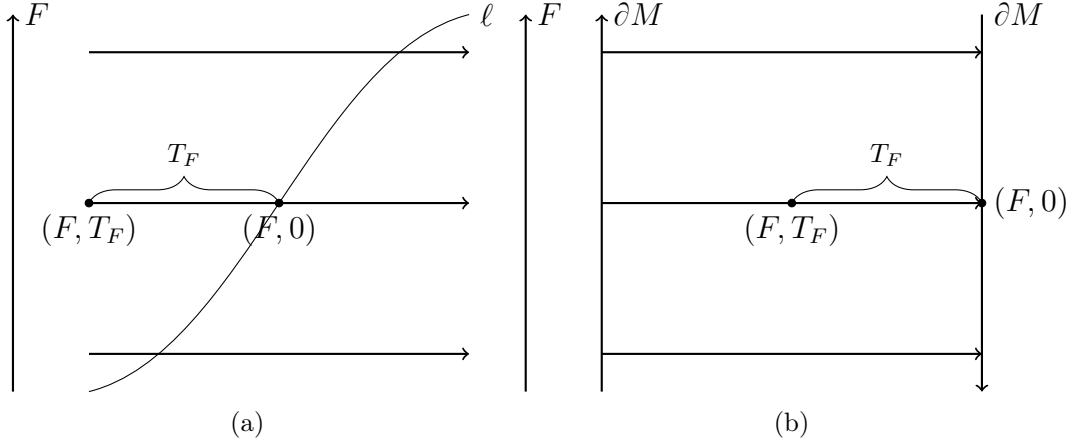
Let $e\subset\Gamma_F^d$ be a dashed edge. The formula $(F,T_F)\mapsto(G,T_G)$ defines a symplectomorphism from the interior of 
$
\pi_F^{-1}(e)
$ 
to the interior of 
$
\pi_G^{-1}(\phi(e)).
$ 
The condition $\phi_*\mu_F=\mu_G$ guarantees that the periods of the functions $T_F$ and $T_G$ coincide and hence the symplectomorphism is well-defined. Now let $e'\subset \Gamma_F^s$ be a solid edge. Let 
$
\mathit{\ell} \subset M
$
be a smooth oriented curve  which is transversal to the level sets of the function $F.$ We also assume it does not intersect
the singular levels of the function $F$; and the function $F$ is strictly increasing along the curve $\mathit{\ell}.$ 
Then, as above, we obtain a symplectomorphism from the interior of 
$
\pi_F^{-1}(e')
$ 
to the interior of 
$
\pi_G^{-1}(\phi(e')).
$ 
By applying the same procedure to all edges of the graph $\Gamma_F$ we obtain a symplectomorphism 
$$\Phi_1\colon \pi^{-1}[\Gamma_F\setminus V(\Gamma_F)]\to\pi^{-1}[\Gamma_G\setminus V(\Gamma_G)]$$
such that $\Phi_1^*F=G$ and $\pi_G\circ\Phi_1=\phi\circ\pi_F.$
\par
Now let $O$ be a singular point for the function $F$ or its restriction $F|_{\partial M}.$ Then there is only one way to define the image of $O$:
$$\Phi(O):=\pi_G^{-1}(\phi(\pi_F(O))).$$
Let $(p_F,q_F)$ (respectively, $(p_G,q_G)$) be a chart centered at the point $O$ (respectively, $\Phi(O)$) as in Theorem~\ref{theorem:md}~or~\ref{theorem:md_boundary}. Then the condition $\phi_*\mu_F=\mu_G$ guarantees that the corresponding functions $\lambda_F$ and $\lambda_G$ are the same or they differ by a function flat at the origin. In the latter case it follows from Theorem~\ref{theorem:md}~or~\ref{theorem:md_boundary} that we can replace the chart $(p_F,q_F)$ with a chart $(\tilde{p}_F,\tilde{q}_F)$ such that $\tilde{\lambda}_F=\lambda_G.$ So without loss of generality we may assume that $\lambda_F=\lambda_G.$ Therefore, one can define $\Phi$ in some neighbourhood $U_O$ of $O$ by the formula $$\Phi \colon (p_F,q_F)\mapsto (p_G,q_G).$$ This local symplectomorphism $\Phi$ extends uniquely to a semi-local symplectomorphism 
$$
\Phi\colon \pi_F^{-1}(\pi_F(U))\to \pi_G^{-1}([\phi\circ \pi_F](U)).
$$
Without loss of generality we may assume that $\pi_F^{-1}(\pi_F(U_O))$ is a ``standard'' neighbourhood of the singular level $\pi_F^{-1}(\pi_F(O))$ (see Table~\ref{table:types}), i.e. it is a connected component of the set $\{P\in M\colon \abs{F(P)-F(O)}<\varepsilon\}$ containing the point $O$ and the number $\varepsilon>0$ is sufficiently small so that these ``standard'' neighbourhoods for distinct $O$ are pairwise disjoint.
Denote by $U_{F,\varepsilon}$ the union of all these neighbourhoods.
By applying the same procedure to all singular points of the function $F$ or its restriction $F|_{\partial M}$ we obtain a symplectomorphism 
$$\Phi_2\colon U_{F,\varepsilon}\to U_{G,\varepsilon}.$$
such that $\Phi_2^*F=G$ and $\pi_G\circ\Phi_2=\phi\circ\pi_F.$
\par
So the isomorphism $\phi \colon \Gamma_F \to \Gamma_G$ is lifted to a symplectomorphism 
$$\Phi_1\colon \pi^{-1}[\Gamma_F\setminus V(\Gamma_F)]\to\pi^{-1}[\Gamma_G\setminus V(\Gamma_G)]$$ 
and to a symplectomorphism
$$\Phi_2\colon U_{F,\varepsilon}\to U_{G,\varepsilon}.$$
However, these two symplectomorphisms not necessarily define a global symplectomorphism of the surface $M.$ 
Let $e\subset \Gamma_F^d$ be a dashed edge. Then the intersection 
$
U_{F,\varepsilon}\cap\pi^{-1}(e)
$
is a disjoint union of two rectangles and the ratio 
$
\Phi_2^{-1}\circ\Phi_1
$ 
is a symplectic automorphism of this union preserving each component and also preserving the function $F.$ The only symplectic automorphism of a fibered rectangle is the identity. So 
$
\Phi_2^{-1}\circ\Phi_1 = \Id
$ 
on 
$
U_{F,\varepsilon}\cap\pi^{-1}(e)
$
i.e. the symplectomorphisms $\Phi_2$ and $\Phi_1$ agree with each other on the preimage $\pi_F^{-1}(e)$ of the edge $e.$
Now let $e\subset \Gamma_F^s$ be a solid edge. Then the intersection 
$
U_{F,\varepsilon}\cap\pi^{-1}(e)
$
is a disjoint union of two open cylinders and the ratio 
$
\Phi_2^{-1}\circ\Phi_1
$ 
is a symplectic automorphism of this union preserving each component and also preserving the function $F.$ Any symplectic automorphism of a fibered cylinder is a Hamiltonian automorphism. The same holds for their union. The corresponding Hamiltonian $H_t$ extends (using a bump function) to a smooth function on all of $M$ in such a way that its support is in the preimage of the edge $e$ for all $t\in[0,1].$
Let us denote this globally defined Hamiltonian automorphism by $\Theta.$ Now we have 
$
\Phi_1|_{\pi_F^{-1}(e)}=(\Phi_2 \circ \Theta)|_{\pi_F^{-1}(e)},
$ 
i.e. the symplectomorphisms $\Phi_2\circ\Theta$ and $\Phi_1$ do agree with each other on the preimage $\pi_F^{-1}(e)$ of the edge $e.$ By applying the same procedure to all solid edges of $\Gamma_F$ we obtain a globally defined symplectomorphism $\Phi:M\to M$ such that $\Phi^*F=G$ and $\pi_G\circ\Phi=\phi\circ\pi_F.$ This completes the proof of part $i)$.
\par
Now let us prove the second statement of the theorem. Given a triple $(\Gamma, f,\mu)$ we need to construct a quadruple  $(\tilde{M},\tilde{\pi},\tilde{F},\tilde{\omega})$ such that $\tilde{F}=f\circ \tilde{\pi}$ and $\tilde{\pi}_*\tilde{\omega}=\mu.$
If this is done then $$\int_{\tilde{M}}\tilde{\omega}=\int_{\Gamma}\diff \mu=\int_M\omega$$ and it follows from Moser's theorem~\cite{moser1965volume} that there is a diffeomorphism $\Phi:\tilde{M}\to M$ such that $\Phi^*\omega=\tilde{\omega}$ so that one can take 
$
F=\tilde{F}\circ \Phi^{-1}.
$ 
It follows from~\cite{hladysh2019simple}
that there exists a surface $\tilde{M}$ with a simple Morse function $\tilde{F}$ and a projection $\tilde{\pi}:\tilde{M}\to\Gamma$ such that $\tilde{F}=f\circ \pi.$ It remains to construct a symplectic form $\tilde{\omega}$ such that $\tilde{\pi}_*\tilde{\omega}=\mu.$
Let $O$ be a singular point of the function $\tilde{F}$ or its restriction $\tilde{F}|_{\partial M}.$ It follows from the proofs of Theorem~\ref{theorem:md} and Theorem~\ref{theorem:md_boundary} that there exists a symplectic form $\omega_O\in \Omega^2(M)$ such that 
$\pi_*(\omega_O)|_{U}=\mu|_{U}$ for some neighborhood $U$ of the vertex $\tilde{\pi}(O).$ Using an appropriate partition of unity we construct a symplectic form $\tilde{\omega}$ as a combination of forms $\omega_O,$ such that $\tilde{\pi}_*\tilde{\omega}=\mu$.
\end{proof}
\section{Classification of generic coadjoint orbits of symplectomorphism groups}
\label{section:orbits}
\subsection{From Morse functions to coadjoint orbits} \label{subsection:orbits_to_functions}
Throughout this section, let $(M,\omega)$ be a compact connected symplectic surface with boundary $\partial M.$ By 
$\SDiff(M)$ we denote the Lie group\footnote{
See \cite[Chapter I, Section 1.1]{khesin2008geometry} for details on Lie groups and Lie algebras in an infinite-dimensional setting.
} 
of all symplectomorphisms of $M.$ Note that all elements of $\SDiff(M)$ preserve the boundary $\partial M$ but do 
\emph{not necessarily} preserve the boundary $\partial M$ \emph{pointwise}. The group $\SDiff(M)$ has the Lie algebra $\SVect(M)$ of divergence-free vector fields on $M$ tangent to the boundary $\partial M.$ The regular dual space $\SVect^*(M)$ can be identified with the space of cosets $\Omega^1(M)/d\Omega^0(M)$ (see Appendix). Moreover, the natural action of the group $\SDiff(M)$ on the space of cosets $\Omega^1(M)/\diff\Omega^0(M)$ by means of pull-backs coincides with the coadjoint action of the group of symplectomorphisms $\SDiff(M):$
$$\Ad_\Phi^*[\alpha]=[\Phi^*\alpha],$$
where $\Phi\in\SDiff(M)$ is a symplectomorphism and $\alpha \in\Omega^1(M)$ is a 1-form. 
\par
Define the exterior derivative operator $\diff$ on the space of cosets $\{\alpha+\diff f|f\in C^\infty(M)\}$ by the formula $\diff[\alpha]:=\diff \alpha.$ (This operator is well-defined on the cosets since $\diff(\alpha+\diff f)=\diff\alpha.)$
Consider the following mapping:
    $$\Diff \colon \Omega^1(M)/\diff \Omega^0(M) \to C^\infty(M),$$
defined by taking a \emph{vorticity function} 
$
\diff \alpha / \omega =: \Diff[\alpha].
$ 
It is easy to see that if the boundary $\partial M$ of the surface $M$ is not empty then the mapping $\Diff$ is a surjection. In the case of a closed surface $M$ there is a relation:
$$\int_M \Diff[\alpha]\omega=0$$
and the mapping $\Diff$ is surjective onto the space of zero-mean functions.
\par
Suppose that cosets $[\alpha]$ and $[\beta]$ belong to the same coadjoint orbit of $\SDiff(M)$. Then by definition, there is a symplectomorphism $\Phi$ such that $[\Phi^*\beta]=[\alpha]$ and the following diagram is commutative:
\[
\xymatrix{
[\beta] \ar[r]^{\Phi^*} \ar[d]_{\Diff} & [\alpha] \ar[d]^{\Diff} \\
\Diff[\beta] \ar[r]^{\Phi^*} & \Diff[\alpha]
}
\]
\begin{definition}
A coset $[\alpha]\in \Omega^1(M)/\diff \Omega^0(M)$ is called \emph{simple Morse} if $\Diff[\alpha]$ is a simple Morse functions. A coadjoint orbit $\pazocal O$ is called \emph{simple Morse} if some (and hence every) coset $[\alpha]\in \pazocal O$ is simple Morse. 
\end{definition}
With every simple Morse coset $[\alpha]\in \Omega^1(M)/\diff \Omega^0(M)$ one can associate a measured Reeb graph $\Gamma_{\Diff[\alpha]}.$ If two simple Morse cosets $[\alpha]$ and $[\beta]$ belong to the same coadjoint orbit then the corresponding Reeb graphs are isomorphic.
\par
Suppose that cosets $[\alpha]$ and $[\beta]$ have isomorphic Reeb graphs. Then it follows from Theorem~\ref{classification_of_functions_symplectic} that there exists a symplectomorphism $\Phi$ such that $\Phi^*\Diff[\beta]=\Diff[\alpha].$ Therefore, the 1-form $\Phi^*[\beta]-[\alpha]$ is closed. Since this 1-form is not necessarily exact, the cosets $[\alpha]$ and $[\beta]$ do not necessarily belong to the same coadjoint orbit. Nevertheless, we conclude that the space of coadjoint orbits 
corresponding to the same measured Reeb graph is finite-dimensional and its dimension is at most $\dim \Hom^1(M).$ Throughout this section, unless otherwise stated, all (co)homology groups will be with coefficients in $\R.$
\subsection{Circulation functions on a Reeb graph}
In~\cite{izosimov2016coadjoint} the notion of a circulation function was introduced for the case of closed surfaces. In the case of surfaces with boundary, we need a modification of that definition. Take a point $x\in\Gamma_F^s$ which is not a vertex. Then $\pi^{-1}(x)$ is a circle $C.$ It is naturally oriented as the boundary of the set of smaller values of the function $F$. The integral of a coset $[\alpha]$ over $C$ is well-defined. Thus we obtain a function 
$$
\circulation_{[\alpha]} \colon \Gamma_F^s\setminus V(\Gamma_F^s)\to \mathbb{R},
$$
defined by $\circulation_{[\alpha]}(x)=\int_{\pi^{-1}(x)}\alpha.$
\begin{proposition}[\cite{izosimov2016coadjoint}]
\label{circProperties}
The function $\circulation_{[\alpha]}=\int_{\pi^{-1}(x)}\alpha$ has the following properties.
\begin{enumerate}[label=(\roman*)]  
\item Assume that $x$ an $y$ are two interior points of some edge $e \subset \Gamma_F^s,$ and that $e$ is pointing from $x$ towards $y$. Then $\circulation_{[\alpha]}$ satisfies the Newton-Leibniz formula
$$
\circulation_{[\alpha]}(y)-\circulation_{[\alpha]}(x)=\int_x^y fd\mu
$$
\item for all vertices of $\Gamma^s$ which do not belong to $\Gamma^d$ the function $\circulation_{[\alpha]}$ satisfies the Kirchhoff rule at $v$:
\begin{align}
\label{3valentcirc}
\sum_{{e \to v}} \lim\nolimits_{{x \xrightarrow[]{e} v }} \circulation_{[\alpha]}(x)= \sum_{{e \leftarrow v}} \lim\nolimits_{{x \xrightarrow[]{e} v }} \circulation_{[\alpha]}(x)\,, 
\end{align}
where the notation $e \to v$ stands for the set of edges pointing at the vertex $v$, and ${e \leftarrow v}$ stands for the set of solid edges pointing away from $v$.  
\end{enumerate}
\end{proposition}
Note that the function $f$ on the subgraph $\Gamma^s_F$ can be recovered from the circulation function 
$\circulation$ by the formula: $f= \diff \circulation/\diff \mu$. It follows from Proposition~\ref{circProperties} that the difference $C_{[\alpha]}-C_{[\beta]}$ is as an element of the relative homology group $\Hom_1(\Gamma_F,\Gamma_F^d).$ 
\par
The above properties of the circulation function $\circulation_{[\alpha]}$ make it natural to introduce the following definition of an abstract circulation function.
\begin{definition}\label{def:circ_graph}
{\rm
Let  $(\Gamma, f, \mu)$ be a measured Reeb graph. Any function $\circulation \colon \Gamma^s \setminus V(\Gamma^s) \to \R$ satisfying properties listed in Proposition~\ref{circProperties} is called a \emph{circulation function (an antiderivative)}.
}
\end{definition}
\begin{proposition}\label{circFuncs}
Let $(\Gamma, f, \mu)$ be a measured Reeb graph. 
\begin{longenum}
\item If the subgraph $\Gamma^d$ is not empty, then the pair $(f,\mu)$ on $\Gamma$ admits an antiderivative.
\item 
If the subgraph $\Gamma^d$ is empty, then the pair $(f,\mu)$ on $\Gamma$ admits an antiderivative if and only if
$
\int_{\Gamma}f\diff \mu = 0
$.
\item
If the pair $(f,\mu)$ admits an antiderivative, then the set of antiderivatives of $(f,\mu)$ is an affine space whose underlying vector space is the relative homology group $\Hom_1(\Gamma, \Gamma^d).$
\end{longenum}
\end{proposition}
\begin{proof}
To prove this result one applies Proposition~$3.12$ in~\cite{izosimov2016characterization} to the graph $\Gamma^s,$ with the set of boundary vertices defined as those vertices that belong to $\Gamma^d.$
\end{proof}
\subsection{Auxiliary classification result}
In this subsection we follow~\cite{izosimov2016characterization}. Let $(M,\omega)$ be a symplectic surface with boundary $\partial M.$ Denote by $CB(M)\subset C^\infty(M)$ the space of Morse functions on $M$ constant on the boundary $\partial M,$ and without critical points on the boundary $\partial M.$ Elements of $CB(M)$ are called functions of $CB-$type. 
\begin{definition}
A coset $[\alpha]\in \Omega^1(M)/\diff \Omega^0(M)$ is said to be of
\emph{$CB$-type} 
if 
$
\Diff[\alpha] \in CB(M).
$ 
A coadjoint orbit $\pazocal O$ called to be of
\emph{$CB$-type} 
if some (and hence every) coset $[\alpha]\in \pazocal O$ is of $CB$-type. 
\end{definition}
All definitions from the present paper such as Reeb graph, compatibility conditions, circulation graph, etc. can be modified for the case of functions and cosets of $CB-$type, see details in \cite{izosimov2016characterization}. The result we are interested in can be formulated as follows.
\begin{theorem}[\cite{izosimov2016characterization}]
\label{classification_of_orbitsCB}
Let $M$ be a connected symplectic surface with or without boundary. Then coadjoint orbits of $\SDiff(M)$ of $CB-$type are in one-to-one correspondence with (isomorphism classes of) circulation graphs  
$(\Gamma, f, \mu, \circulation)$  compatible with $M$.
In other words, the following statements hold:
\begin{longenum}
\item For a symplectic surface $M$ and cosets of $CB-$type $[\oneform], [\oneformtwo] \in \SVect^*(M)$ 
the following conditions are equivalent:
\begin{longenum} \item $[\oneform]$ and $ [\oneformtwo]$ lie in the same orbit of the $\SDiff(M)$ coadjoint action;  \item circulation graphs $\Gamma_{[\oneform]}$ and $\Gamma_{[\oneformtwo]}$ corresponding to the cosets $[\oneform]$ and $[\oneformtwo]$ 
are isomorphic.\end{longenum}
\item For each circulation graph $\Gamma$ which is compatible 
with $M$, there exists a generic $[\oneform] \in \SVect^*(M)$ such that  $\Gamma_{[\oneform]} =(\Gamma, f, \mu, \circulation).$
\end{longenum}
\end{theorem}
\subsection{Augmented circulation graph}
In the case of surfaces with boundary circulation functions do not form a complete set of invariants for coadjoint orbits, i.e. the equality $C_{[\alpha]}=C_{[\beta]}$ does not in general  imply that cosets $\alpha$ and $\beta$ belong to the same coadjoint orbit. 
\begin{example}
Consider the disk with two holes from Figure~\ref{fig:dashed_Reeb}(a). In this case there are no circulation functions since there are no solid edges in the Reeb graph. On the other hand, in this case there are no nontrivial symplectomorphisms preserving the function hence the dimension of the space of coadjoint orbits is equal to the first Betti number of the surface, i.e. it is equal to two.
\end{example}
It turns out that it is possible to define some additional invariants: integrals of cosets over certain cycles associated with the pair $(M,F)$ in an invariant way. 
\par
There is a unique way to lift each edge $e\subset \Gamma_F^d$ to a smooth oriented (and diffeomorphic to a segment) curve  $\tilde{e}\subset \partial M$ such that 
\begin{longenum}
\item $\pi(\tilde{e})=e;$
\item for each $x\in e\setminus \partial e$ the regular $F$-level $\pi^{-1}(x)$ is pointed in the direction of the curve $\tilde{e}.$
\end{longenum}
We define the subset $\tilde{E}_F\subset M$ to be the union
$$
\tilde{E}_F:=\bigcup_{e\in E(\Gamma^d_F)}\tilde{e}.
$$ 
We also define the subset $\tilde{V}_F\subset M$ to be 
$$
\tilde{V}_F:=\pi^{-1}(V(\Gamma^d_F)\setminus \partial \Gamma^d_F)
$$
where $\partial \Gamma^d_F$ is the set of boundary vertices (i.e. vertices of types I, III, or V) of the graph $\Gamma^d_F.$ And, finally, define the subset $\tilde{\Gamma}_F$ to be the union of $\tilde{E}_F$ and $\tilde{V}_F$ (see Figure \ref{fig:embedded_graph}).
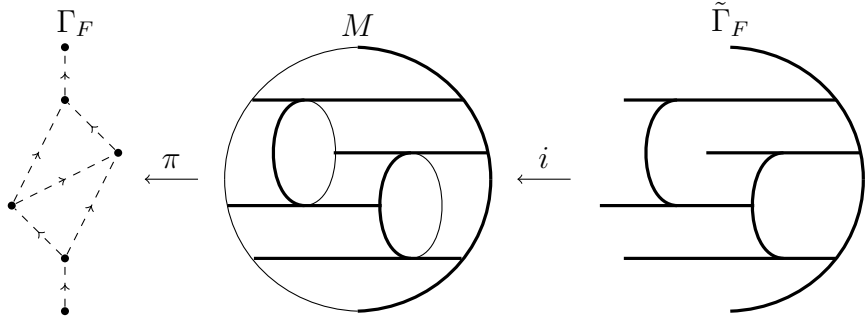
\begin{figure}
    \centering
    \begin{tikzpicture}[scale = 0.7]
    \tikzstyle{subj} = [circle, minimum width=3pt, fill, inner sep=0pt]
    \draw (5.7,5) node[above] {$\Gamma_F$};
    \draw (11,5) node[above] {$M$};
    \draw (18,5) node[above] {$\tilde{\Gamma}_F$};
    \node[subj] (nodeA) at (5.5,0) {};
    \node[subj] (nodeB) at (5.5,1) {};
    \node[subj] (nodeC) at (4.5,2) {};
    \node[subj] (nodeD) at (6.5,3) {};
    \node[subj] (nodeE) at (5.5,4) {};
    \node[subj] (nodeF) at (5.5,5) {};
    \draw[dashed, 
    decoration={markings, mark=at position 0.5 with {\arrow{>}}},
    postaction={decorate}]
    (nodeA) -- (nodeB);
    \draw[dashed, 
    decoration={markings, mark=at position 0.5 with {\arrow{>}}},
    postaction={decorate}]
    (nodeB) -- (nodeC);
    \draw[dashed, 
    decoration={markings, mark=at position 0.5 with {\arrow{>}}},
    postaction={decorate}] 
    (nodeB) -- (nodeD);
    \draw[dashed, 
    decoration={markings, mark=at position 0.5 with {\arrow{>}}},
    postaction={decorate}] 
    (nodeC) -- (nodeE);
    \draw[dashed, 
    decoration={markings, mark=at position 0.5 with {\arrow{>}}},
    postaction={decorate}] 
    (nodeC) -- (nodeD);
    \draw[dashed, 
    decoration={markings, mark=at position 0.5 with {\arrow{>}}},
    postaction={decorate}] 
    (nodeD) -- (nodeE);
    \draw[dashed, 
    decoration={markings, mark=at position 0.5 with {\arrow{>}}},
    postaction={decorate}] 
    (nodeE) -- (nodeF);
    \draw[<-] (7,2.5) -- (8,2.5) node[above, midway] {$\pi$};
    \draw[<-] (14,2.5) -- (15,2.5) node[above, midway] {$i$};
    fill [opacity = 0.1] (1+10, 2.5) ellipse (2.5 and 2.5);
    \draw[very thick] (9.05,1) -- (12.95,1);
    \draw[very thick] (10-1.45,2) -- (11.45,2);
    \draw[very thick] (10.55,3) -- (13.45,3);
    \draw[very thick] (9.02,4) -- (13,4);
    \draw[very thick] (11,0) to [out=2, in=-90] (13.5,2.5);
    \draw[very thick] (13.5,2.5) to [out=90, in=-2] (11,5);
    \draw (11,0) to [out=180-2, in=-90] (11-2.5,2.5);
    \draw (11-2.5,2.5) to [out=90, in=180+2] (11,5);
    \draw (12,1) to [out=2, in=-2] (12,3);
    \draw[very thick] (12,1) to [out=180-2, in=180+2] (12,3);
    \draw (10,2) to [out=2, in=-2] (10,4);
    \draw[very thick] (10,2) to [out=180-2, in=180+2] (10,4);
    \draw[very thick] (7+9,1) -- (7+13,1);
    \draw[very thick] (7+10-1.45,2) -- (7+11.45,2);
    \draw[very thick] (7+10.55,3) -- (7+13.45,3);
    \draw[very thick] (7+9,4) -- (7+13,4);
    \draw[very thick] (7+11,0) to [out=2, in=-90] (7+13.5,2.5);
    \draw[very thick] (7+13.5,2.5) to [out=90, in=-2] (7+11,5);
    \draw[very thick] (7+12,1) to [out=180-2, in=180+2] (7+12,3);
    \draw[very thick] (7+10,2) to [out=180-2, in=180+2] (7+10,4);
\end{tikzpicture}
    \caption{An illustration to the definition of the graph $\tilde{\Gamma}_F.$}
    \label{fig:embedded_graph}
\end{figure}
The set $\tilde{\Gamma}_F$ is a topological graph embedded into the surface $M.$ We denote by $i$ the inclusion
$\tilde{\Gamma}_F\xhookrightarrow{} M.$
\begin{lemma}\label{lemma_hom_eqv}
The map $\pi\circ i\colon\tilde{\Gamma}_F\to\Gamma_F^d$ is a homotopy equivalence.
\end{lemma}
\begin{proof}
Consider the graph $\tilde{\Gamma}_F/\tilde{V}_F$ obtained from $\tilde{\Gamma}_F$ by contracting each connected component of a singular $F$-level in $\tilde{\Gamma}_F$ to a point. Denote by $p$ the projection $\tilde{\Gamma}_F\to\tilde{\Gamma}_F/\tilde{V}_F.$ The map $p$ is a homotopy equivalence since each singular $F$-level in $\tilde{\Gamma}_F$ is connected and simply connected. The map
$i\circ \pi$ factors (in a unique way) through $\tilde{\Gamma}_F/\tilde{V}_F,$ i.e. there exists a unique map $\tilde{\pi}:\tilde{\Gamma}_F/\tilde{V}_F\to\Gamma_F^d$ such that
$
\pi\circ i=\tilde{\pi}\circ p \circ i.
$
The map $\tilde{\pi}$ is a homeomorphism. We conclude that the map $p\circ i$ is a homotopy equivalence as a composition of the inclusion $i,$ the homotopy equivalence $p,$ and the homeomorphism $\tilde{\pi}.$ 
\end{proof}
Let $[\alpha]\in\Omega^1(M)/\diff \Omega^0(M)$ be a coset of a one-form.
There is a natural way to define the restriction $i^*[\alpha]\in\Hom^1(\tilde{\Gamma}_F).$ 
First, we define the restriction $i^*\alpha$ as a one-cochain such that 
$
i^*\alpha(e):=\int_{e}\alpha
$
for each edge $e\subset \tilde{\Gamma}_F.$
Now we take $i^*[\alpha]:=[i^*\alpha].$ The cohomology class $i^*[\alpha]$ is well-defined since each exact one-form $\diff f$ restricts to the exact one-cochain  $i^*\diff f.$ It follows from Lemma~\ref{lemma_hom_eqv} that 
$
i^*\circ \pi^*\colon\Hom^1(\Gamma_F)\to\Hom^1(\tilde{\Gamma}_F)
$ 
is an isomorphism. Hence with each coset 
$
[\alpha]\in \Omega^1(M)/\diff \Omega^0(M)
$
we can also associate an element $
\xi_{[\alpha]}\in\Hom^1(\Gamma_F)
$ 
defined by the formula
$
\xi_{[\alpha]}:=(i^*\circ \pi^*)^{-1}(i^*[\alpha]).
$
Next, we generalize the notion of a circulation graph from~\cite{izosimov2016coadjoint}. 
\begin{definition}
A measured Reeb graph $(\Gamma, f,\mu)$ endowed with a circulation function $\pazocal{C}$ and an element 
$
\xi \in \Hom^1(\Gamma^d)
$
is called a \emph{augmented circulation graph} 
$
(\Gamma ,f,\mu, \circulation,\xi).
$
\end{definition}
We demonstrated above that with each coset $[\alpha]$ one can associate an augmented circulation graph $\Gamma_{[\alpha]}.$
Two augmented circulation graphs are isomorphic if they
are isomorphic as measured Reeb graphs, and the isomorphism between them preserves all additional data. An augmented circulation graph 
$
(\Gamma ,f,\mu, \circulation,\xi)
$
is \emph{compatible} with a symplectic surface $(M,\omega)$ if the corresponding measured Reeb graph $(\Gamma ,f,\mu)$ is compatible with $(M,\omega)$ (see Definition~\ref{def:compatible}).
\subsection{Coadjoint orbits of symplectomorphism groups}
\begin{theorem}\label{thm4}\label{classification_of_orbits}
Let $(M,\omega)$ be a connected symplectic surface with or without boundary. Then generic coadjoint orbits of $\SDiff(M)$ are in one-to-one correspondence with (isomorphism classes of) augmented circulation graphs  
$(\Gamma, f, \mu, \circulation)$  compatible with $M$.
In other words, the following statements hold:
\begin{longenum}
\item For a symplectic surface $M$ and generic cosets $[\oneform], [\oneformtwo] \in \SVect^*(M)$ 
the following conditions are equivalent:
\begin{longenum} \item $[\oneform]$ and $ [\oneformtwo]$ lie in the same orbit of the $\SDiff(M)$ coadjoint action;  
\item augmented circulation graphs $\Gamma_{[\oneform]}$ and $\Gamma_{[\oneformtwo]}$ corresponding to the cosets $[\oneform]$ and $[\oneformtwo]$ 
are isomorphic.\end{longenum}
\item For each augmented circulation graph $\Gamma$ which is compatible with $M$, there exists a generic $[\oneform] \in \SVect^*(M)$ such that  $\Gamma_{[\oneform]} =(\Gamma, f, \mu, \circulation, \xi)$.
\end{longenum}
\end{theorem}
\begin{corollary}
The space of coadjoint orbits of the group $\SDiff(M)$ corresponding to the same measured Reeb graph $(\Gamma,f,\mu)$ is a finite-dimensional affine space and its dimension is $\dim \Hom_1(\Gamma,\Gamma^d) + \dim \Hom_1(\Gamma^d).$
\end{corollary} 
\begin{remark}
It follows from the long exact sequence for the pair $(\Gamma,\Gamma^d)$ that 
$$
\dim \Hom_1(\Gamma,\Gamma^d) + \dim \Hom_1(\Gamma^d)=\dim \Hom_1(\Gamma) - \dim \Hom_0(\Gamma^d)+1.
$$
Therefore, the space of coadjoint orbits of the group $\SDiff(M)$ corresponding to the same measured Reeb graph $(\Gamma,f,\mu)$ has dimension $\dim \Hom_1(\Gamma)$ in the case when the subgraph $\Gamma^d$ is connected. 
\end{remark}
\begin{example}\label{example_coadjoint}
Consider the torus with one boundary component from Figure~\ref{fig:example_reeb} with the height function $F$ on it, and the corresponding Reeb graph $\Gamma_F$. In this case $\Hom_1(\Gamma_F^d)=0$ and $\Hom_1(\Gamma_F,\Gamma_F^d)=1.$ Therefore, the corresponding space of coadjoint orbits is one-dimensional. 
\end{example}
Before we proceed with the proof of Theorem~\ref{classification_of_orbits} let us formulate and prove two lemmas.
\begin{lemma}\label{dashedSurfaceCohomology}
Let $M$ be a connected oriented surface with non-empty boundary, and let $F$ be a simple Morse function on $M$. Then 
$$
\dim \Hom_1(M^d_F) = \dim \Hom_1(\Gamma^d_F)+ \dim \Hom_0(\Gamma^s_F\cap\Gamma^d_F).
$$
\end{lemma}
\begin{proof}
Let $\tilde{M}$ be the smooth surface obtained from the surface $M^d_F$ by 
contracting each circle in $M^d_F\cap M^s_F$ to a point. It is clear that 
$$
\dim \Hom_1(M^d_F)=\dim \Hom_1(\tilde{M})+ \dim \Hom_0(\Gamma^s_F\cap\Gamma^d_F).
$$
Let $p$ be the canonical projection $M\to \tilde{M}.$ The function $F$ descends to a simple Morse function $\tilde{F}\colon \tilde{M}\to \R$ such that
$
F=\tilde{F}\circ p.
$
The Reeb graph $\Gamma_{\tilde{F}}$ consists only of dashed edges, and it is coincides with $\Gamma^d_F.$ Then the surface $\tilde{M}$ is homotopy equivalent to graph $\Gamma_{\tilde{F}}.$ Therefore, 
\begin{equation}
\begin{aligned}
& \dim \Hom_1(M^d_F)=\dim \Hom_1(\tilde{M})+ \dim \Hom_0(\Gamma^s_F\cap\Gamma^d_F)= \dim \Hom_1(\Gamma_{\tilde{F}})+ \dim \Hom_0(\Gamma^s_F\cap\Gamma^d_F)\\
& =\, \dim \Hom_1(\Gamma^d_F)+ \dim \Hom_0(\Gamma^s_F\cap\Gamma^d_F).
\end{aligned}
\end{equation}
\end{proof}
\begin{lemma}\label{graphCohomology}
Let $M$ be a connected oriented surface possibly with boundary, and let $F$ be a simple Morse function on $M$. 
Assume that $[\gamma] \in  \Hom^1(M)$ is such that the integral of $\gamma$ over 
any $F$-level vanishes, and $\xi_{[\gamma]}$ is a zero element in $\Hom^1(\Gamma_F^d).$ 
Then there exists a $C^\infty$ function $H \colon M \to \R$ (with zero restriction on the surface $M^d_F$) such that the one-form $H\diff F$ is closed, and its cohomology class 
is equal to $[\gamma]$. Moreover, $H$ can be chosen in such a way that the ratio
$H/F$ is a smooth function.
\end{lemma}
\begin{proof}
Denote by $i_d$ the inclusion $M_F^d\cap M_F^s \xhookrightarrow{} M^d_F,$ and denote by $\pi_d$ the restriction of the projection $\pi\colon M\to \Gamma_F$ on the surface $M^d_F.$ Note that the homomorphism
$
(\pi_d)_*\colon \Hom_1(M_F^d) \to \Hom_1(\Gamma_F^d)
$
is a surjection, and 
$
\Imm (i_d)_* \subset \Ker (\pi_d)_*.
$
It follows from Lemma~\ref{dashedSurfaceCohomology} that 
$$
\dim \Hom_1(M^d_F) = \dim \Hom_1(\Gamma^d_F)+ \dim \Hom_0(\Gamma^s_F\cap\Gamma^d_F).
$$
Hence the image of the homomorphism $(\pi_d)_*$ coincides with the kernel of the homomorphism $(\pi_d)_*.$
From above we conclude that the homomorphism 
$
\pi_*\colon \Hom^1(\Gamma_F^d) \to \Hom^1(M_F^d)
$
is an injection, and
$
\Imm (\pi_d)^* = \Ker (i_d)^*.
$
\par
Denote by $i$ the inclusion $M_F^d \xhookrightarrow{} M_F.$ 
Since the integral of $[\gamma]$ over any connected component of any closed $F$-level vanishes 
and $\xi_{[\gamma]}$ is a zero element in $\Hom^1(\Gamma_F^d),$
the cohomology class $[i^*\gamma]$ is a zero element in $\Hom^1(M_F^d).$
Consider the long exact cohomology sequence for the pair $(M_F, M^d_F)\colon$
\begin{alignat*}{2}
0 &\rightarrow \Hom^{0}(M_F^d) &\rightarrow \Hom^{0}(M_F)\rightarrow \Hom^1(M_F,M^d_F) \rightarrow \Hom^1(M_F) &\rightarrow \Hom^1(M^d_F) \rightarrow 0.
\end{alignat*}  
The cohomology class $[\gamma]$ on $M$ belongs to the kernel of the homomorphism 
$
i^*\colon \Hom^1(M_F)\to \Hom^1(M^d_F).
$ 
Hence it belongs to the image of the homomorphism
$
\Hom^1(M_F,M^d_F) \to \Hom^1(M_F),
$
i.e. there exists a one-form $\tilde{\gamma}$ such that 
$[\tilde{\gamma}]=[\gamma]$ and $\tilde{\gamma}|_{M^d_F}=0.$ 
\par
Denote by $\pi_s$ the restriction of the projection $\pi\colon M\to \Gamma_F$ on the surface $M^s_F.$ The homomorphism
$
(\pi_s)_*\colon \Hom_1(M_F^s) \to \Hom_1(\Gamma_F^s)
$
is a surjection, and its kernel is generated by those homology classes which are homologous to regular $F$-levels. From above we conclude that the homomorphism
$
(\pi_s)^*\colon \Hom^1(\Gamma_F^s) \to \Hom^1(M_F^s)
$
is an injection, and 
$
\Imm(\pi_s)^*=\Ann\Ker(\pi_s)_*
$ 
where 
$$
\Ann\Ker(\pi_s) := \left \{ \omega \in  \Hom^1(\Gamma_F^s) \mid 
\omega(c)=0 \iff c \in \Ker(\pi_s)_* \right \}.
$$
Therefore, there exists a one-cochain 
$
\alpha\in \Hom^1(\Gamma^s_F)
$ 
of the form
$
\alpha=\sum_{e\in E(\Gamma^s_F)}\alpha_e e^*
$
such that $[\tilde{\gamma}]=(\pi_s)^*[\alpha].$
\par
Recall that the function $f$ is the pushforward of the function $F$ to the graph $\Gamma_F.$ Consider a continuous function $ h \colon \Gamma_F \to \R$ such that 
\begin{longenum}
\item it is a smooth function of $f$ in a neighborhood of each point $x \in \Gamma_F$;
\item it vanishes whenever $x$ is sufficiently close to a vertex;
\item $h|_{\Gamma^d_F}=0;$
\item for each edge $e$, we have
$$
\alpha (e) = \int_e \! h\diff f.
$$
\end{longenum}
Obviously, such  a function does exist. Now, lifting $ h$ to $M$, we obtain a smooth function $H$ with 
the desired properties.
\end{proof}
\begin{proof}[Proof of Theorem~\ref{classification_of_orbits}]
Let us prove the first statement. The implication (a) $\Rightarrow$ (b) is immediate, 
so it suffices to prove the implication (b) $\Rightarrow$ (a). 
Let $\phi \colon \Gamma_{[\oneform]} \to \Gamma_{[\oneformtwo]}$ be an isomorphism 
of augmented circulation graphs. By Theorem~\ref{classification_of_functions_symplectic}, $\phi$ can be lifted to a symplectomorphism 
$ \Phi \colon M \to M$ that  maps the function $F = \Diff[\oneform]$ 
to the function $G = \Diff[\oneformtwo]$. Therefore, the $1$-form $\oneformthree$ defined by
$$
\oneformthree =  \Phi^*\oneformtwo - \oneform 
$$
is closed. 
\par
Assume that $ \Psi \colon M \to M $ is a symplectomorphism which maps the function $F$ to itself and is isotopic to the identity. Then the composition $\widetilde\Phi = \Phi\circ\Psi^{-1}$ maps $F$ to $G$, and
$$
[\widetilde\Phi^*\oneformtwo - \oneform] = [\Phi^*\oneformtwo - \Psi^*\oneform] = [\gamma] - [\Psi^*\oneform - \oneform].
$$
We claim that $\Psi$ can be chosen in such a way that $\widetilde\Phi^*\oneformtwo - \oneform$ is exact, i.e.  one has the equality of the cohomology classes
$$
 [\Psi^*\oneform - \oneform] = [\gamma].
$$
Moreover, we show that there exists a time-independent symplectic vector field  $X$ that
preserves $F$ and satisfies
\begin{align}\label{homotopy2}
 [\Psi_t^*\oneform - \oneform] = t[\gamma]\,,
\end{align}
where $\Psi_t$ is the flow of $X$. 
Differentiating \eqref{homotopy2} with respect to $t$, we get in the left-hand side
$$
 [\Psi_t^*L_X\oneform] =  [L_X\oneform] =  [i_X\diff \oneform] = [F\cdot i_X \omega]\,,
$$
since the form $L_X \oneform$ is closed and $\Psi_t^*$ does not change its cohomology class. Thus
\begin{align}\label{homology2}
[F\cdot i_X \omega] = [\gamma].
\end{align}
Since $\Phi$ preserves the circulation function, the integrals of $\oneformthree$ over all connected components of $F$-levels vanish.
In addition, $\xi_{[\Phi^*\alpha]}=\xi_{[\alpha]}.$
Therefore, by Lemma \ref{graphCohomology}, there exists a smooth function $H$ such that
$$
[\gamma] = [H\diff F].
$$
Now we set 
$$
X :=  \frac{H}{F} \,\omega^{-1}\diff F.
$$
It is easy to see that the vector field $X$ is zero on $M^d,$ symplectic, preserves the levels of $F$, and satisfies the equation \nolinebreak \eqref{homology2}. Therefore, its phase flow map $\Psi=\Psi_1$  has the required properties.
\par
Now let us prove the second statement. It follows from Theorem~\ref{classification_of_functions_symplectic} that there exists a symplectic surface $(M,\omega)$ and a simple Morse function $F\colon M\to\R$ such that $\Gamma_F=\Gamma.$ Consider the surface $M^s$ and the restriction $F|_{M^s}$ of the function $F$ to the surface $M^s.$ The restriction $F|_{M^s_F}$ is a Morse function, and it is constant on the boundary $\partial M^s_F$ since it is formed by some of closed $F$-levels. However, it is not necessarily a function of $CB$-type since it has hyperbolic critical points on the boundary whenever the graph $\Gamma_F$ has vertices of type $V.$ In order to apply the Theorem~\ref{classification_of_orbitsCB}
we need to `cut out' from $M^s$ these hyperbolic critical points. Let $v\in \Gamma_F$ be a vertex of type $V,$ let $e\subset \Gamma_F$ be the only solid edge incident to $v,$ and also let $u\in\Gamma^s_F$ be the only other vertex adjacent to $e.$ The edge $e$, with endpoints $\{v,u\}$ can be uniquely subdivided into two edges, say $e_{v\to w}$ and $e_{w\to u}$, connecting to a new vertex $w$ such that $\mu(e_{v\to w})=\mu(e_{w\to u}).$ After that we cut out the edge $e_{v\to w}.$ Denote by $\Gamma'$ the (abstract) measured Reeb graph obtained by applying the above procedure to all vertices of type $V$ in the graph $\Gamma_F$ (see Figure \ref{fig:graph_subdivision}).
\begin{figure}
    \centering
    \begin{tikzpicture}[scale = 1]
    \tikzstyle{subj} = [circle, minimum width=3pt, fill, inner sep=0pt]
    \draw[->] (1.5,-0.5) -- (2.5,-0.5) node[above, midway] {};
    \node[subj, label=right:$v$] (nodeA) at (0,0) {};
    \node[subj, label=right:$u$] (nodeB) at (1,-1) {};
    \node (nodeC) at (-1,-1) {};
    \node (nodeD) at (-1,1) {};
    \node (nodeP) at (2,-2) {};
    \node (nodeQ) at (0,-2) {};
    \draw (nodeA) -- (nodeB);
    \draw[dashed] (nodeA) -- (nodeC);
    \draw[dashed] (nodeA) -- (nodeD);
    \draw (nodeB) -- (nodeP);
    \draw (nodeB) -- (nodeQ);
    \draw (0.5,-0.5) node[below left] {$e$};
    \node[subj, label=right:$v$] (nodeA1) at (4+0,0) {};
    \node[subj, label=right:$u$] (nodeB1) at (4+1,-1) {};
    \node[subj, label=right:$w$] (nodeF) at (4+0.5,-0.5) {};
    \node (nodeC1) at (4-1,-1) {};
    \node (nodeD1) at (4-1,1) {};
    \node (nodeP1) at (4+2,-2) {};
    \node (nodeQ1) at (4+0,-2) {};
    \draw (nodeA1) -- (nodeB1);
    \draw[dashed] (nodeA1) -- (nodeC1);
    \draw[dashed] (nodeA1) -- (nodeD1);
    \draw (nodeB1) -- (nodeP1);
    \draw (nodeB1) -- (nodeQ1);
\end{tikzpicture}
    \caption{The illustration to the construction of the graph $\Gamma'$. The vertex $w$ subdivides the edge $e$ with endpoints $\{v,u\}$ into the edges $e_{v\to w}$ and $e_{w\to u}.$}
    \label{fig:graph_subdivision}
\end{figure}
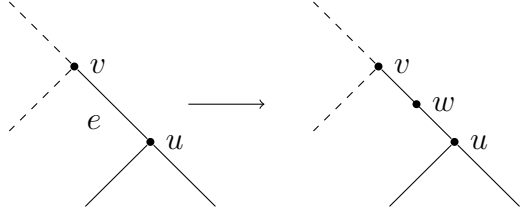 
Denote by $M'\subset M^s_F$ the preimage 
$
\pi^{-1}(\Gamma')
$
It is clear from the above that the restriction 
$F|_{\tilde{M}}$ is a function of $CB$-type. Therefore, it follows from Theorem~\ref{classification_of_orbitsCB} that there exists a one-form $\oneform_0$ on $M'$ such that 
$
\circulation_{[\oneform_0]}=\circulation|_{\Gamma'}.
$ 
It is clear that the form $\oneform_0$ can be extended (along the cylinders $M^s\setminus M'$) on all of $M^s$ in such a way that 
$
\circulation_{[\oneform_0]}=\circulation|_{\Gamma}.
$ 
On the other hand, there exists a one-form $\oneform_1$ on $M$ such that
$
[(i_d)^*\alpha_1]=[(i_s)^*\alpha_0]
$
and
$
\xi_{[\alpha]}=\xi
$
since 
$$
\dim \Hom_1(M^d_F) = \dim \Hom_1(\Gamma^d_F)+ \dim \Hom_1(M^d_F\cap M^s_F).
$$
Using an appropriate partition of unity we construct a one-form $\alpha$ (as a combination of one-forms 
$\alpha_0$ and $\alpha_1$) such that
$
\circulation_{[\alpha]}=\circulation_{[\alpha_0]}=\circulation
$
and
$
\xi_{[\alpha]}=\xi_{[\alpha_1]}=\xi.
$
Hence the augmented circulation graph $\Gamma_{[\alpha]}$ coincides with $\Gamma.$
\end{proof}
\section{Conclusion}\label{section:Final_remarks}
In this paper we have classified simple Morse functions on symplectic surfaces and generic coadjoint orbits of symplectomorphism groups of surfaces.
This allowed us to completely resolve the questions posed in~\cite{izosimov2017classification}.
The answer to Problem~$5.5$ on reconstruction of a surface with boundary from its Reeb graph was given in the work~\cite{hladysh2019simple}. As described in Section~\ref{section:global} the idea is to add a cyclic order for the dashed edges incident to II- or IV-vertices to the structure of an abstract Reeb graph in order to reconstruct the corresponding surface. Table~\ref{table:types} gives an answer to Problem~$5.6$ on measure asymptotics on the graph. The required by Problem~$5.7$ compatibility conditions are described in Definition~\ref{def:compatible}. And finally, Theorem~\ref{classification_of_orbits} describes the required by Problem~$5.8$ additional invariants for coadjoint orbits in case of surfaces with boundary.
\par
One should mention two other relevant but not overlapping with us classification results for symplectic surfaces:
\begin{enumerate}[a)]
    \item Dufour, Molino, and Toulet classified in~\cite{dufour1994classification} simple Morse fibrations on closed symplectic surfaces under the action of symplectic diffeomorphisms.
    \item Bolsinov~\cite{bolsinov1995smooth} and Kruglikov~\cite{kruglikov1997exact} classified Hamiltonian vector fields tangent to the boundary on surfaces up to the action of arbitrary diffeomorphisms.
\end{enumerate}
It would be interesting to extend those classifications for surfaces with boundary. (Note that in~\cite{kruglikov1997exact} in the contrast with the present work Hamiltonian functions are assumed to be constant on the boundary.)
It also would be very interesting to classify Morse functions and Morse orbits for the action of 
\begin{enumerate}[a)]
    \item the group $\Ham(M)$ of Hamiltonian diffeomorphisms of a surface $M;$ 
    \item the connected component $\SDiff_0(M)$ of the identity in the group $\SDiff(M)$ for the case of surfaces $M$ with boundary. 
\end{enumerate}
This would generalise the corresponding results of~\cite{izosimov2016coadjoint} and the present work to these important subgroups of the symplectomorphism groups.
\begin{appendices}
\section{Euler's equation and coadjoint orbits}\label{section:euler}
In this Appendix we describe following~\cite{izosimov2016coadjoint} the hydrodynamical motivation of the above classification problems. Consider a symplectic surface $(M,\omega)$ with boundary $\partial M.$ We denote by $\SDiff(M)$ the Lie group of all symplectomorphisms of $M,$ and by $\SVect(M)$ the corresponding Lie algebra of divergence-free vector fields on $M.$ A linear functional $I$ on $\SVect(M)$ is called \emph{regular} if there exists a smooth 1-form $\xi_I$ such that the value of $I$ on a vector field $v$ is the pairing between $\xi_I$ and $v:$
$$I(v)=\int_M \xi_I(v)\omega.$$ 
The space $\SVect^*(M)$ of regular functionals on $\SVect(M)$ is a dense subset in the space of all continues linear functionals on $\SVect(M).$ It turns out that the space of regular functionals $\SVect^*(M)$ can be identified with the space of cosets 
$\Omega^1(M)/d\Omega^0(M),$ since exact $1$-forms give zero functionals on divergence-free vector fields. Moreover, the natural action of the group $\SDiff(M)$ on the space of cosets $\Omega^1(M)/d\Omega^0(M)$ by means of pull-backs coincides with the coadjoint action of the group of symplectomorphisms $\SDiff(M)$. The proof of this fact can be found in~\cite{arnold1999topological} (see Section I.8). More information about infinite-dimensional Lie groups can be found in~\cite{khesin2008geometry}.
\par
Now let us fix a Riemannian metric $(\cdot,\cdot)$ on the surface $M$ such that the corresponding area form coincides with the symplectic form $\omega.$ The  motion  of an inviscid incompressible fluid  on  $M$  is described by  the Euler equation
\begin{equation}\label{idealEuler}
\partial_t v+\nabla_v v = -\nabla p
\end{equation}
describing an evolution of a  divergence-free velocity field $v$ of a fluid flow in $M$, where 
${\rm div}\, v=0$ and the field $v$ is tangent to the boundary $\partial M.$ 
The pressure function $p$ entering the Euler equation is defined uniquely modulo 
an additive constant by this equation along with the divergence-free constraint 
on the velocity $v$.
\par
The metric $(\cdot,\cdot)$ allows us to
identify the (smooth parts of) the Lie algebra and its dual by means of the so-called inertia operator:
given a vector field $v$ on $M$  one defines the 1-form $\alpha=v^\flat$ 
as the pointwise inner product with vectors of the velocity field $v$:
$v^\flat(W) := (v,W)$ for all $W\in T_x M.$
The Euler equation \eqref{idealEuler} rewritten on 1-forms is
$$\partial_t \alpha+L_v \alpha=-dP\,$$
for the 1-form $\alpha=v^\flat$ and an appropriate function $P$ on $M$.
In terms of the cosets of 1-forms $[\alpha]=\{\alpha+df\,|\,f\in C^\infty(M)\}\in \Omega^1(M) / \diff \Omega^0(M)$, the Euler equation looks as follows: 
\begin{equation}\label{1-forms}
\partial_t [\alpha]+L_v [\alpha]=0
\end{equation}
on the dual space $\mathfrak g^*$, where $L_v$ is the Lie derivative along the field $v$. 
\par
The Euler equation \eqref{1-forms} shows that the coset of 1-forms $[\alpha]$ evolves by an area-preserving change of coordinates, i.e. it remains in the same coadjoint orbit in $\mathfrak g^*$. This is why
invariants of coadjoint orbits of cosets $[\alpha]$ describe first integrals, called Casimirs, of the Euler equation, and their complete classification is important in many areas of ideal fluid dynamics.
\end{appendices}
\bibliographystyle{plain}
\bibliography{kirillov_orbits}
\addcontentsline{toc}{section}{References}
\end{document}